\documentclass[10pt,a4paper]{amsart}
\usepackage{amsaddr}
\usepackage{amssymb,amsthm,amsmath,amsfonts,amscd}
\usepackage{graphicx}
\usepackage{url}
\usepackage{color}
\usepackage[active]{srcltx}
\usepackage[matrix,arrow]{xy}
\usepackage{mathrsfs}
\usepackage{enumerate}
\usepackage{amsopn}
\usepackage{bbm}
\usepackage{ulem}
\usepackage{cite}
\usepackage{hyperref}
\allowdisplaybreaks[1]
\usepackage[english]{babel}
\usepackage{bbm}
\usepackage{stmaryrd}

\newcounter{cprop}[section]

\newtheorem{definition}[cprop]{Definition}
\newtheorem{remark}[cprop]{Remark}
\newtheorem{lemma}[cprop]{Lemma}

\newtheorem{proposition}[cprop]{Proposition}
\newtheorem{theorem}[cprop]{Theorem}

\usepackage{mdef}

\title[Finite speed of propagation for SPME]{Finite speed of propagation for stochastic porous media equations.}
\author[B. Gess]{Benjamin Gess}
\date{\today}
\email{gess@math.tu-berlin.de}
\address{Institut f\"ur Mathematik, Technische Universit\"at Berlin (MA 7-5)\\
Stra\ss{}e des 17. Juni 136, 10623 Berlin, Germany}
\thanks{{\bf Acknowledgements:} The author would like to thank Michael R\"ockner for valuable discussions and comments.}
\keywords{stochastic partial differential equations, stochastic porous medium equation, finite speed of propagation, hole-filling, free boundary.}
\subjclass[2010]{37L55, 60H15; 76S05, 37L30}

\begin{document}

\begin{abstract}
  We prove finite speed of propagation for stochastic porous media equations perturbed by linear multiplicative space-time rough signals. Explicit and optimal estimates for the speed of propagation are given. The result applies to any continuous driving signal, thus including fractional Brownian motion for all Hurst parameters. The explicit estimates are then used to prove that the corresponding random attractor has infinite fractal dimension.
\end{abstract}

\maketitle


\section{Introduction}

In this paper we prove finite speed of propagation for solutions to stochastic porous media equations (SPME) driven by linear multiplicative space-time rough signals, i.e.\ to equations of the form
\begin{equation}\label{eqn:SPME}\begin{split}
    d X_t         &= \D \left(|X_t|^m \sgn(X_t) \right)dt + \sum_{k=1}^N f_k X_t \circ dz^{(k)}_t, \text{ on } \mcO_T, \\
    X(0)          &= X_0, \text{ on } \mcO,
\end{split}\end{equation}
with homogeneous Dirichlet boundary conditions on a bounded, smooth domain $\mcO \subseteq \R^d$, $m \in (1,\infty)$, rough driving signals $z^{(k)} \in C([0,T];\R)$
and diffusion coefficients $f_k \in C^\infty(\bar\mcO)$. 
We assume the number of signals $N$ to be finite and high regularity for $f_k$ for simplicity only. In fact, the proofs only require $\sum_{k=1}^\infty f_k(\xi)z_t^{(k)} \in C([0,T];C^2(\bar\mcO))$. The stochastic Stratonovich integral $\circ$ occurring in \eqref{eqn:SPME} is informal but justified by a transformation technique and stability results analyzed in detail in \cite{G11c,G11c-2}.

Recently, a hole-filling property for SPME driven by multiplicative space-time Brownian noise has been shown in \cite{BR12}, which may be seen as an important step towards proving finite speed of propagation. However, no explicit control on the rate of growth of the support of the solution could be established, which made it impossible to deduce finite speed of propagation. In the present paper, we prove explicit (and locally optimal) estimates on the speed of hole-filling and thus deduce finite speed of propagation for SPME. Moreover, we will completely remove the non-degeneracy assumption on the noise as it was conjectured to be possible in \cite{BR12}, which allows to analyze the dependence of the speed of propagation on the strength of the noise. In particular, we prove convergence to the deterministic, optimal estimates when the noise-intensity converges to zero (cf.\ Remark \ref{rmk:noise_intens} below). In \cite{BR12} restrictions on the dimension $d$ and on the order of the nonlinearity $m$ had to be supposed for technical reasons and it was conjectured that these could be completely removed. In the present paper we prove that this indeed is the case.

Our methods are purely local and thus apply without change to the homogeneous Cauchy-Dirichlet problem to \eqref{eqn:SPME} on not necessarily bounded domains $\mcO \subseteq \R^d$, as soon as the problem of unique existence of corresponding solutions is solved. Since up to now this problem remains open, we restrict to bounded domains for simplicity (cf.,\ however, Remark \ref{rmk:unbdd} below).

The stochastic case is contained in our setup by choosing $z^{(k)}$ to be given as paths of some continuous stochastic process. Therefore, our results yield purely pathwise results for the stochastic case. However, due to the explicit form of our estimates, moment estimates also immediately follow.

In the deterministic case it is well-known that the attractor corresponding to 
\begin{equation}\label{eqn:det_pert}
  d X_t = \D \left(|X_t|^m \sgn(X_t) \right)dt + \l X_t dt,
\end{equation}
with Dirichlet boundary conditions has infinite fractal dimension iff $\l > 0$ (cf.\ \cite{EZ08}). Generally speaking, it highly depends on the drift of an SPDE as well as on the type of random perturbation, whether the noise has a regularizing effect on the long-time dynamics of the unperturbed system. 

In \cite{G11b} it has been shown that sufficiently non-degenerate additive Wiener noise stabilizes the dynamics of \eqref{eqn:det_pert} in the sense that the random attractor consists of a single random point and thus is zero dimensional. Moreover, it is well-known that multiplicative It\^o noise may stabilize the long-time dynamics due to the It\^o correction term. For example, this has been realized in \cite{CCLR07} in case of the Chafee-Infante equation perturbed by spatially homogeneous, linear multiplicative It\^o noise. The more intriguing case of space-time, linear multiplicative It\^o noise has been analyzed in \cite{BDPR12} for fast diffusion equations (cf.\ also the references therein), where a regularizing effect due to the It\^o correction term has been observed in \cite[Theorem 3.5]{BDPR12}. 

This correction term is absent in the case of linear multiplicative Stratonovich noise. In this spirit, it has been shown in \cite{CCLR07} that spatially homogeneous, linear multiplicative Stratonovich noise does not have any regularizing effect on the long-time behavior of the Chafee-Infante equation. On the other hand, each linear PDE with non-negative, self-adjoint drift having negative trace (possibly $-\infty$) may be stabilized by linear multiplicative space-time Stratonovich noise (cf.\ \cite{CR04}). For these reasons, it is an intriguing question, whether including linear multiplicative space-time Stratonovich noise in \eqref{eqn:det_pert} stabilizes the long-time behavior, or whether the random attractor associated to 
\begin{equation}\label{eqn:pert_SPME_intro}
   d X_t = \D \left(|X_t|^m \sgn(X_t) \right)dt + \l X_t dt + \sum_{k=1}^N f_k X_t \circ dz^{(k)}_t,
\end{equation}
remains infinite dimensional. Based on the explicit bounds on the rate of propagation obtained in this paper, we prove lower bounds for the Kolmogorov $\ve$-entropy of the random attractor corresponding to \eqref{eqn:pert_SPME_intro} and thus conclude that the random attractor remains infinite dimensional.

The SPME \eqref{eqn:SPME} with driving signals $z^{(k)}$ given as paths of independent Brownian motions $\b^{(k)}$ has been intensively studied in the recent history (cf.\ e.g.\ \cite{DPR04,K06,DPRRW06,RRW07,RW08,BDPR08,BDPR08-2,BDPR09,G11} and references therein). The construction of a random dynamical system (RDS) associated to \eqref{eqn:SPME} and the proof of existence of a corresponding random attractor has been given in \cite{G11c,G11c-2}. In case of porous media equations (PME) perturbed by additive noise, the existence of a random attractor has been shown in \cite{BGLR10} and has subsequently been generalized to more general additive perturbations \cite{GLR11} and spatially rougher noise \cite{G11b}.

The sublinear, fast diffusion case ($m \in [0,1)$) exhibits completely different propagation properties. In particular, finite speed of propagation does not hold for fast diffusion equations, but the positivity set of non-trivial solutions will cover the hole domain of definition after an arbitrarily small timespan (cf.\ \cite{V07} and references therein). On the other hand, solutions to the fast diffusion equation become extinct in finite time (cf.\ \cite{V07} for the deterministic case, \cite{BDPR09-2,RW11} for the stochastic case).

In the following let $\mcO \subseteq \R^d$ be a bounded domain with smooth boundary $\Sig := \partial\mcO$. For $T > 0$ we define the space-time domain $\mcO_T := [0,T] \times \mcO$, the lateral boundary $\Sig_T := [0,T] \times \partial\mcO$ and the parabolic boundary $\mcP_T := \Sig_T \cup (\{T\} \times \mcO)$. Let $\vartheta$ be the surface measure on $\Sig$ and $\nu$ be the outward pointing normal vector to $\Sig$. By $C^0(\mcO)$ we denote the space of continuous functions on $\mcO$ and by $C^{m,n}(\mcO_T)$ the space of continuous functions on $\mcO_T$ with $m$ continuous derivatives in time and $n$ continuous derivatives in space. $C_c^{m,n}(\mcO_T)$ is the subspace of all compactly supported functions in $C^{m,n}(\mcO_T)$. We define $C^n(\bar\mcO)$, $C^{m,n}(\bar\mcO_T)$ to be the spaces obtained by restricting the functions in $C^n(\R^d)$, $C^{m,n}([0,T] \times \R^d)$ onto $\bar\mcO$. Moreover, we define $H$ to be the dual of the first order Sobolev space with zero boundary $H_0^1(\mcO)$. For two non-empty subsets $A,B$ of a metric space $(E,d)$ we define $\dist(A,B) := \inf\{d(a,b)|\ a\in A,\ b\in B\}$. If $X$ is a Banach space, then $L^p_{loc}((0,T];X)$ denotes the space of all $X$-valued functions $f$ such that $f \in L^p([\tau,T];X)$ for all $\tau \in (0,T]$. As usual in probability theory we often denote the time-dependency of functions by a subscript $X_t$ rather than by $X(t)$ in order to keep the equations at a bearable length. 

Let us start by recalling the finite speed of propagation properties for deterministic PME
\begin{equation}\label{eqn:PME}
  \partial_t u = \D \Phi(u),
\end{equation}
where for simplicity of notation we have set $\Phi(u) := |u|^{m}\sgn(u)$. Finite speed of propagation for deterministic PME has been known for a long time and was first proved in \cite{OKY58}. For a more detailed study on interfaces for the one dimensional case we refer to \cite{V84}. Our main reference for the deterministic PME and main source of inspiration for the stochastic case will be \cite{V07} where a beautiful account on the propagation and expansion properties for deterministic PME is given.

\begin{definition}[Notions of solutions for \eqref{eqn:PME}]
  \begin{enumerate}
   \item A function $u \in L^1_{loc}(\mcO_T)$ with $\Phi(u) \in L^1_{loc}(\mcO_T)$ is said to be a local, very weak subsolution to \eqref{eqn:PME} if 
     \begin{equation*}
       \int_{\mcO_T} u \partial_r \eta\ d\xi dr \ge - \int_{\mcO_T} \Phi(u) \D \eta\ d\xi dr, 
    \end{equation*}
    for all non-negative $\eta \in C^{1,2}_c(\mcO_T)$. 
   \item If, in addition, $u \in L^1(\mcO_T)$ with $\Phi(u) \in L^1(\mcO_T)$ and there are functions $u_0 \in L^1(\mcO)$ and $\Phi(g) \in L^1(\Sig_T)$ such that
    \begin{equation*}
      \int_{\mcO_T} u \partial_r \eta\ d\xi dr + \int_\mcO u_0\eta_0\ d\xi  \ge - \int_{\mcO_T} \Phi(u) \D \eta\ d\xi dr + \int_{\Sig_T} \Phi(g) \partial_\nu \eta\ d\vartheta dr, 
    \end{equation*}
  for all non-negative $\eta \in C^{1,2}(\bar\mcO_T)$ with $\eta_{| \mcP_T} = 0$, then $u$ is said to be a very weak subsolution to the (inhomogeneous) Dirichlet problem to \eqref{eqn:PME} with initial condition $u_0$ and boundary value $g$.
  
    \item If $\Phi(u) \in L^2([0,T];H^1_0(\mcO))$ then $u$ is said to be a (local) weak subsolution to the homogeneous Dirichlet problem to \eqref{eqn:PME}. 
    \item If $\Phi(u) \in L^2_{loc}((0,T];H^1_{0}(\mcO))$ then $u$ is said to be a generalized (local) weak subsolution to the homogeneous Dirichlet problem to \eqref{eqn:PME}. 
  \end{enumerate} 

  Analogous definitions are used for (local) very weak supersolutions. (Local) very weak solutions to \eqref{eqn:PME} are functions that are supersolutions and subsolutions simultaneously.
\end{definition}

We note that each essentially bounded, generalized weak solution $u$ is a generalized weak solution to \eqref{eqn:PME} on each smooth subdomain $K \subseteq \mcO$ with initial data ${u_0}_{|K}$ and boundary data $\Phi(g) = \Phi(u)$ in the sense of traces.

The proof of finite speed of propagation is a direct consequence of the so-called hole-filling problem

\begin{lemma}[Deterministic hole-filling, \cite{V07}, Lemma 14.5]\label{lemma:det_hole}
    Let $\xi_0 \in \R^d$, $T,R > 0$ and $u \in C((0,T)\times B_R(\xi_0))$ be an essentially bounded, non-negative, very weak subsolution to \eqref{eqn:PME} with vanishing initial value $u_0$ on $B_R(\xi_0)$ and boundary value $g$ satisfying $H:= \|g\|_{L^\infty([0,T]\times \partial B_R(\xi_0))} < \infty$. Define $C_{det} = \frac{m-1}{2dm(m-1)+4m}$ and
     $$T_{det} :=  R^2 \frac{C_{det}}{H^{m-1}}.$$
   Then $u(t)$ vanishes in $B_{R_{det}(t)}(\xi_0)$ for all $t \in [0,T_{det} \wedge T]$, where
     $$R_{det}(t) = R - \sqrt{t} \left(\frac{H^{m-1}}{C_{det}}\right)^\frac{1}{2}.$$
\end{lemma}

For boundary value $g$ given as $g \equiv H$ for some $H >0$, the bound on the rate of hole-filling from Lemma \ref{lemma:det_hole} is optimal (cf.\ \cite[p.\ 339]{V07}). 

From the hole-filling Lemma one may deduce

\begin{theorem}[Deterministic finite speed of propagation, \cite{V07}, Theorem 14.6]\label{thm:det_prop}
  Let $u \in C((0,T)\times\mcO)$ be an essentially bounded, non-negative, very weak subsolution to the homogeneous Dirichlet problem to \eqref{eqn:PME} and set $H = \|u\|_{L^\infty(\mcO_T)}$. Then 
  \begin{enumerate}
   \item For every $s \in [0,T]$ and every $h > 0$ there is a time-span $T_h > 0$ such that
      $$ \supp(u_{s+t}) \subseteq B_h(\supp(u_s)), \quad \forall t \in [0,T_h \wedge (T-s)]. $$
    More precisely, $T_h$ is given by 
      $$T_h :=  h^2 \frac{C_{det}}{H^{m-1}}.$$
   \item For every $s \in [0,T]$
      $$ \supp(u_{s+t}) \subseteq B_{\sqrt{t}\left(\frac{H^{m-1}}{C_{det}}\right)^\frac{1}{2}}(\supp(u_s)), \quad \forall t \in [0,T-s].$$
  \end{enumerate}
\end{theorem}
\begin{proof}
  For each non-negative $u_0 \in L^\infty(\R^d) \cap L^1(\R^d)$ there is a unique non-negative, essentially bounded, weak solution $u_C \in C([0,T];L^1(\R^d))$ to the Cauchy problem for \eqref{eqn:PME} (cf.\ \cite[Theorem 9.3]{V07}). This in turn is a weak supersolution to the Cauchy-Dirichlet problem on $\mcO$. Since $t\mapsto \|u_C(t)\|_\infty$ is non-increasing, we have $\|u_C\|_{L^\infty(\mcO_T)} = \|u\|_{L^\infty(\mcO_T)}$. Without loss of generality one may thus assume that $u$ is a solution to the Cauchy problem, which simplifies the argument since no difficulties at the boundary appear. Noticing that $u$ in particular is an essentially bounded, very weak solution on each $B_R(\xi_0)$, the claim becomes a direct consequence of Lemma \ref{lemma:det_hole}.
\end{proof}

\section{Real-valued linear multiplicative noise}\label{sec:real}

We start the analysis of the stochastically perturbed case by the much simpler situation of spatially homogeneous noise, i.e.\ we consider the homogeneous Dirichlet problem to
\begin{equation}\label{eqn:real_valued}
  dX_t = \D \Phi(X_t)dt + \sum_{k=1}^N f_k X_t \circ dz_t^{(k)},\quad \text{on } \mcO_T,
\end{equation}
where $f_k \in \R$ are $\R$-valued constants and $\mcO \subseteq \R^d$ is as before. We will prove below that \eqref{eqn:real_valued} reduces to the deterministic PME \eqref{eqn:PME} by rescaling and a random transformation in time. Since the bounds on the rate of propagation are known to be optimal in the deterministic case, we deduce optimal bounds for the case of spatially homogeneous perturbations. 

Let $\mu_t = -\sum_{k=1}^N f_k z_t^{(k)} $ and $Y_t := e^{\mu_t}X_t$. Then (informally)
\begin{equation}\label{eqn:homog_transf}
  \partial_t Y_t = e^{-(m-1)\mu_t} \D\Phi(Y_t), \quad \text{on } \mcO_T.
\end{equation} 
Solutions to \eqref{eqn:real_valued} are then defined by the reverse transformation, i.e.\ a function $X$ is a solution to \eqref{eqn:real_valued} with initial value $X_0 \in L^1(\mcO)$ and boundary value $g$ iff $Y_t := e^{\mu_t}X_t$ is a solution to \eqref{eqn:homog_transf} with initial value $Y_0 := e^{\mu_0}X_0$ and boundary value $e^{\mu}g$. 

In \cite{G11c} it has been shown that this transformation can be made rigorous if the signals $z^{(k)}$ are given as paths of continuous semimartingales or are of bounded variation. In addition, in case of continuous driving signals, solutions to \eqref{eqn:real_valued} were obtained in \cite{G11c} as limits of approximating solutions driven by smoothed signals $z^{(\d)} \in C^\infty([0,T];\R^N)$ with $z^{(\d)} \to z$ in $C([0,T];\R^N)$.

We set $F(t) := \int_0^t e^{-(m-1)\mu_r} dr \in C^1(\R_+;\R_+)$. Since $F$ is strictly increasing we may define $G(t) := F^{-1}(t)$ to be the inverse of $F$ and $u_t := Y_{G(t)}$. An informal computation suggests
\begin{equation}\label{eqn:det_PME_2}
   \partial_t u_t = \D\Phi(u_t), \quad \text{on } \mcO_T.
\end{equation}
A rigorous justification of this temporal transformation can easily be given by considering an artificial viscosity approximation, i.e. $\partial_t u^{(\ve)} = \D\Phi(u_t^{(\ve)}) + \ve \D u_t^{(\ve)}$. Local uniform continuity of $u^{(\ve)}$ (cf.\ \cite{DB83,V07}) allows to pass to the limit pointwisely and thus implies the claim. 

Vice versa, solutions $X$ to \eqref{eqn:real_valued} can be expressed by solutions to \eqref{eqn:det_PME_2} via:
\begin{equation}\label{eqn:homo_trans}
   X_t := e^{-\mu_t}u_{F(t)}.
\end{equation}
Lemma \ref{lemma:det_hole} implies
\begin{proposition}[Hole-filling for spatially homogeneous noise]\label{prop:real_valued}
    Let $\xi_0 \in \R^d$, $T,R > 0$ and $X \in C((0,T)\times B_R(\xi_0))$ be an essentially bounded, non-negative, very weak subsolution to \eqref{eqn:real_valued} with vanishing initial value $X_0$ on $B_R(\xi_0)$ and boundary value $g$ satisfying $H:= \|e^\mu g\|_{L^\infty([0,T]\times \partial B_R(\xi_0))} < \infty$. Define 
       $$T_{stoch} :=  F^{-1}\left(R^2\frac{C_{det}}{H^{m-1}}\right),$$
    where $F(t) := \int_0^t e^{-(m-1)\mu_r} dr$. 

    Then $X_t$ vanishes in $B_{R_{stoch}(t)}(\xi_0)$ for all $t \in [0,T_{stoch} \wedge T]$, where
       $$R_{stoch}(t) = R - \sqrt{F(t)} \left(\frac{H^{m-1}}{C_{det}}\right)^\frac{1}{2}.$$
\end{proposition}

As pointed out above, the rates and constants given in Proposition \ref{prop:real_valued} are optimal. Analogously, bounds on the rate of expansion of the support of solutions to \eqref{eqn:real_valued} may be derived from \eqref{eqn:homo_trans} and Theorem \ref{thm:det_prop}.

\begin{remark}\label{rmk:homog_local_time}
  For $T \approx 0$ we have $\mu_t \approx \mu_0$ on $[0,T]$ and thus 
  $$ \sqrt{F(t)} = \left(\int_0^t e^{-(m-1)\mu_r}dr\right)^\frac{1}{2} \approx e^{-\frac{(m-1)}{2}\mu_0} \sqrt{t}, \quad \text{on } [0,T]$$ 
and $H = \|e^{\mu}g\|_{L^\infty([0,T] \times \partial B_R(\xi_0))} \approx e^{\mu_0}\|g\|_{L^\infty([0,T] \times \partial B_R(\xi_0))}$. Thus
     $$R_{stoch}(t) \approx R -  \sqrt{t} \left(\frac{\|g\|_{L^\infty([0,T] \times \partial B_R(\xi_0))}^{m-1}}{C_{det}}\right)^\frac{1}{2}, \quad \forall t \in [0,T_{stoch} \wedge T].$$
 Consequently, we recover the deterministic rate of expansion for small times $T \approx 0$.
\end{remark}

\section{Linear multiplicative space-time noise}

We now turn to the case of SPME perturbed by spatially inhomogeneous noise \eqref{eqn:SPME}. Since spatially homogeneous noise is contained as a special case, the precise bounds derived in the last section will serve as optimal bounds for the inhomogeneous case. Let
\begin{equation*}
  \mu_t(\xi) := -\sum_{k=1}^N f_k(\xi) z^{(k)}_t.
\end{equation*}
As in the case of spatially homogeneous noise, solutions to \eqref{eqn:SPME} are defined via the transformation $Y_t := e^{\mu_t}X_t$ which (informally) leads to the transformed equation (first studied in \cite{BDPR09-2,BR11b})
\begin{equation}\label{eqn:trans_inhomog}\begin{split}
  \partial_t Y_t         &= e^{\mu_t} \D \Phi(e^{-\mu_t} Y_t), \text{ on } \mcO_T \\
   Y(0)                   &= Y_0 = e^{\mu_0}X_0, \text{ on } \mcO,
\end{split}\end{equation}
with homogeneous Dirichlet boundary conditions. Note that $\mu_t$ now depends on the spatial variable $\xi \in \mcO$. As before, this transformation can be made rigorous if the driving signals $z^{(k)}$ are given as paths of continuous semimartingales or are of bounded variation. In case of continuous driving signals this notion of solution is justified in a limiting sense via approximation of the driving signal (cf.\ \cite{G11c}). 

Similar results and methods as presented in this section may be applied to the more general equation
\begin{equation*}\begin{split}
  \partial_t Y_t         &= \rho_1 \D \Phi(\rho_2 Y_t), \text{ on } \mcO_T \\
   Y(0)                   &= Y_0, \text{ on } \mcO,
\end{split}\end{equation*}
with $\rho_1, \rho_2 \in C^{0,2}(\mcO_T)$ and zero Dirichlet boundary conditions. For simplicity and in order to derive locally optimal estimates we restrict to equations of the form \eqref{eqn:trans_inhomog} and postpone the treatment of the more general case to the Appendix \ref{sec:gen_expansion}.

The unique existence of weak solutions to \eqref{eqn:trans_inhomog} for bounded initial data has been given in \cite{G11c} in the following sense

\begin{definition}[weak \& very weak solutions for \eqref{eqn:trans_inhomog}]\label{def:weak_soln}\ 
 \begin{enumerate}    
  \item A function $Y \in L^1_{loc}(\mcO_T)$ with $\Phi(e^{-\mu}Y) \in L^1_{loc}(\mcO_T)$ is called a local, very weak subsolution to \eqref{eqn:trans_inhomog} if
      \begin{equation*}
        \int_{\mcO_T} Y \partial_r \eta\ d\xi dr \ge -\int_{\mcO_T} \Phi(e^{-\mu} Y) \D (e^{\mu} \eta )\ d\xi dr, 
      \end{equation*}      
  for all non-negative $\eta \in C^{1,2}_c(\mcO_T)$. 
  \item If, in addition, $Y \in L^1(\mcO_T)$ with $\Phi(e^{-\mu}Y) \in L^1(\mcO_T)$ and there are functions $Y_0 \in L^1(\mcO)$ and $\Phi(g) \in L^1(\Sig_T)$ such that
      \begin{align*}\label{ra_m:eqn:very_weak_rough_transformed}
        \int_{\mcO_T} Y \partial_r \eta\ d\xi dr + \int_\mcO Y_0 \eta_0\ d\xi 
        \ge &-\int_{\mcO_T} \Phi(e^{-\mu} Y) \D (e^{\mu} \eta )\ d\xi dr \\
        &+ \int_{\Sigma_T} \Phi(e^{-\mu} g) \partial_\nu (e^{\mu} \eta) d\vartheta dr, 
      \end{align*}      
  for all non-negative $\eta \in C^{1,2}(\bar\mcO_T)$ with $\eta_{|\mcP_T} = 0$ then $Y$ is said to be a very weak subsolution to the (inhomogeneous) Dirichlet problem to \eqref{eqn:trans_inhomog} with initial condition $Y_0$ and boundary value $g$. 

  \item If $\Phi(e^{-\mu}Y) \in L^2([0,T];H^{1}_0(\mcO))$ then $Y$ is said to be a (local) weak subsolution to the homogeneous Dirichlet problem to \eqref{eqn:trans_inhomog}.
  \item If $\Phi(e^{-\mu}Y) \in L^2_{loc}((0,T];H^{1}_{0}(\mcO))$ then $Y$ is said to be a generalized (local) weak subsolution to the homogeneous Dirichlet problem to \eqref{eqn:trans_inhomog}.
  \end{enumerate}
  Analogous definitions are used for (local) very weak supersolutions. (Local) very weak solutions to \eqref{eqn:trans_inhomog} are functions that are supersolutions and subsolutions simultaneously.
\end{definition}

It is easy to see that every essentially bounded, generalized weak solution $Y$ is a generalized weak solution on each smooth subdomain $K \subseteq \mcO$ with initial condition ${Y_0}_{|K}$ and boundary data $\Phi(g) = \Phi(Y)$ in the sense of traces.

As outlined in the beginning of this section, solutions to \eqref{eqn:SPME} are now defined via the transformation $Y_t = e^{\mu_t}X_t$, i.e.
\begin{definition}[weak \& very weak solutions for \eqref{eqn:SPME}] 
   A function $X \in L^1_{loc}(\mcO_T)$ is said to be a (local, generalized, very) weak sub/supersolution to \eqref{eqn:SPME} with initial condition $X_0$ and boundary value $g$, if $Y_t = e^{\mu_t}X_t$ is a (local, generalized, very) weak sub/supersolution to \eqref{eqn:trans_inhomog} with initial condition $Y_0 = X_0 e^{\mu_0}$ and boundary value $g e^{\mu}$.
\end{definition}

In order to prove finite speed of propagation it will turn out to be sufficient to suppose $X$ to be an essentially bounded, very weak solution to the Dirichlet problem corresponding to \eqref{eqn:SPME}. In fact, more is known:
\begin{proposition}[\hskip-1pt\cite{G11c}, Theorem 1.4, Theorem 1.12, Theorem 1.17]\label{prop:u_ex}
  Let $X_0 \in L^1(\mcO)$. Then 
  \begin{enumerate}
   \item There is a generalized weak solution $X \in C([0,T];L^1(\mcO))$ to the homogeneous Dirichlet problem for \eqref{eqn:SPME} satisfying $X \in C((0,T]\times \mcO)$ and $X \in L^\infty_{loc}((0,T];L^\infty(\mcO))$. 
   \item If $X_0 \ge 0$ a.e.\ on $\mcO$, then $X \ge 0$ a.e.\ on $\mcO$.
   \item If $X_0 \in L^\infty(\mcO)$ then, in addition, $X \in C([0,T];H) \cap L^\infty(\mcO_T)$ and $\Phi(X) \in L^2([0,T];H_0^1(\mcO))$. In particular, $X$ is a weak solution to \eqref{eqn:SPME}.
  \end{enumerate}
  The generalized weak solution $X$ to \eqref{eqn:SPME} from (i) is unique.
\end{proposition}

The proof of finite speed of propagation will rely on local comparison to supersolutions. We now present the required comparison result for essentially bounded, very weak solutions to the inhomogeneous Dirichlet problem.

\begin{theorem}[Comparison for very weak solutions]\label{thm:comp}
  Let $X^{(1)},X^{(2)}$ be essentially bounded sub/supersolutions to \eqref{eqn:SPME} with initial conditions $X^{(1)}_0 \le X^{(2)}_0$ and boundary data $g^{(1)} \le g^{(2)}$, a.e.\ in $\mcO$ respectively. Then,
    $$X^{(1)} \le X^{(2)},\quad\text{a.e.\ in } \mcO.$$
  In particular, essentially bounded, very weak solutions to \eqref{eqn:SPME} are unique.
\end{theorem}

The proof of a more general version of Theorem \ref{thm:comp} may be found in the Appendix \ref{sec:gen_comp}.

\subsection{Finite speed of propagation}\label{sec:finite_speed} We are going to prove bounds on the speed of propagation for \eqref{eqn:SPME}, based on estimates for the rate of hole-filling as in the deterministic case. As we have seen in Section \ref{sec:real}, the optimal bounds on the rate of collapse of balls have to depend on the driving signal. Since the perturbation now is spatially dependent, we expect worse estimates than in Proposition \ref{prop:real_valued}. 

On the other hand, since $\xi \mapsto \mu_t(\xi)$ is continuous and thus $\mu_t(\xi) \approx \mu_t(\xi_0)$ on small balls $B_R(\xi_0)$, locally in space the rate of expansion should be given as in Proposition \ref{prop:real_valued} with $\mu_r \equiv \mu_r(\xi_0)$. This line of thought leads to optimal bounds on the rate of collapse of asymptotically small balls, proven in Theorem \ref{thm:hole_filling_local_space} below.

Moreover, due to the continuity of $t \mapsto \mu_t(\xi)$ we have $\mu_t(\xi) \approx \mu_0(\xi)$ on small time intervals $[0,T]$. Therefore, we expect to recover the optimal bounds from the deterministic case at least for asymptotically small times $T$, which indeed is proven in Theorem \ref{thm:hole_filling_local_time} below. In case of spatially homogeneous perturbations this has been observed in Remark \ref{rmk:homog_local_time}.

\begin{theorem}[Hole-filling theorem for small balls]\label{thm:hole_filling_local_space}
   Let $\xi_0 \in \R^d$, $T,R  > 0$ and $X \in C((0,T] \times B_R(\xi_0))$ be an essentially bounded, non-negative, very weak subsolution to \eqref{eqn:SPME} with vanishing initial value $X_0$ on $B_R(\xi_0)$ and boundary value $g$ satisfying $H := \|e^\mu g\|_{L^\infty([0,T] \times \partial B_R(\xi_0))}  < \infty$. Define $F(t) := \int_0^t e^{-(m-1)\mu_r(\xi_0)}dr$ and
     $$  T_{stoch} :=  F^{-1}\left(R^2 \frac{C_{det}}{H^{m-1}} C_R \right),$$
   where $R \mapsto C_R$ is a continuous, non-increasing function with $\lim_{R \downarrow 0} C_R = 1$. 

   Then $X_t$ vanishes in $B_{R_{stoch}(t)}(\xi_0)$ for all $t \in [0,T_{stoch} \wedge T]$, where
     $$R_{stoch}(t) = R - \sqrt{F(t)} \left(\frac{H^{m-1}}{C_{det}}\right)^\frac{1}{2} C_R^{-\frac{1}{2}}.$$
  
  Note that for $R \approx 0$ we recover the optimal rate from Proposition \ref{prop:real_valued} with $\mu_r \equiv \mu_r(\xi_0)$.
\end{theorem}
\begin{proof}
  Since $X$ is a very weak subsolution to \eqref{eqn:SPME} with initial value $X_0 \equiv 0$ and boundary value $g$, $Y := e^\mu X$ is a very weak subsolution to \eqref{eqn:trans_inhomog} with initial value $Y_0 \equiv 0$ and boundary value $e^\mu g$. 

  For $\xi_1 \in B_R(\xi_0)$, $\td T \in (0,T]$, $r \in (0,\dist(\xi_1,\partial B_R(\xi_0))]$ we construct an explicit supersolution to \eqref{eqn:trans_inhomog} in $[0,\td T] \times B_r(\xi_1)$. Let
    $$W(t,\xi,\xi_1) := \td C |\xi-\xi_1|^\frac{2}{m-1}\left(F(\td  T)-F(t)\right)^{\frac{-1}{m-1}}, \quad t \in [0,\td T), \xi \in B_r(\xi_1),$$
  where $\td C$ will be chosen below (only depending on $R$) and $F$ is as in Section \ref{sec:real} for noise frozen at $\xi_0$, i.e.
    $$F(t) := \int_0^t e^{-(m-1)\mu_r(\xi_0)} dr.$$
  We compute:
  \begin{align*}
    \partial_t W(t,\xi,\xi_1) = \frac{1}{m-1} \td C |\xi-\xi_1|^\frac{2}{m-1} \left(F(\td  T)-F(t)\right)^{\frac{-m}{m-1}}e^{-(m-1)\mu_t(\xi_0)}, 
  \end{align*}
  for all $(t,\xi) \in [0,\td T) \times B_r(\xi_1)$ and 
  \begin{align*}
    &\D \left(e^{-\mu_t(\xi)} W(t,\xi,\xi_1)\right)^m \\
    &\hskip0pt \le \frac{\td C^m}{(m-1)C_{det}} \left(F(\td  T)-F(t)\right)^\frac{-m}{m-1} |\xi-\xi_1|^\frac{2}{m-1}\\
    &\hskip15pt \left( e^{-m \mu_{t}(\xi)} + \frac{2(m-1)}{d(m-1)+2} |\nabla e^{-m \mu_{t}(\xi)}| r + (m-1)C_{det} r^2 |\D e^{-m \mu_{t}(\xi)}| \right) \\
    &\hskip0pt \le \frac{\td C^m}{(m-1)C_{det}}  \left(F(\td  T)-F(t)\right)^\frac{-m}{m-1} |\xi-\xi_1|^\frac{2}{m-1} e^{-m \mu_{t}(\xi)}\\
    &\hskip15pt \left( 1 + \frac{2(m-1)m}{d(m-1)+2}|\nabla \mu_{t}(\xi)| r + m (m-1) C_{det} r^2 (m|\nabla \mu_t(\xi)|^2+\D |\mu_t(\xi)|) \right) \\
     &\le \frac{\td C^m}{(m-1)C_{det}} \left(F(\td T)-F(t)\right)^\frac{-m}{m-1} |\xi-\xi_1|^\frac{2}{m-1} e^{-m \mu_{t}(\xi)} \\
       &\hskip15pt\Big(1+C(d,m)R(1+R) \|\mu\|_{C^{0,2}([0,\td T]\times B_R(\xi_0))} \Big), 
  \end{align*}
  for all $(t,\xi) \in [0,\td T) \times B_r(\xi_1)$.
  We conclude that 
    $$  \partial_t W(t,\xi,\xi_1) \ge e^{\mu_t(\xi)} \D \left(e^{-\mu_t(\xi)} W(t,\xi,\xi_1)\right)^m $$
  on $[0,\td T) \times B_r(\xi_1)$ if
    $$  e^{-(m-1)\|\mu(\xi_0)-\mu(\cdot)\|_{C^0([0,T]\times B_R(\xi_0))}} \ge \frac{\td C^{m-1}}{C_{det}}  \Big(1+C(d,m)R(1+R)\|\mu\|_{C^{0,2}([0,T]\times B_R(\xi_0))}\Big),$$
  which is satisfied if we choose $\td C^{m-1} = C_{det} C_R$ with
  \begin{align*}
    C_R
    &:= \frac{e^{-(m-1)\|\mu(\xi_0)-\mu(\cdot)\|_{C^0([0,T]\times B_R(\xi_0))}}}{1+C(d,m)R(1+R)\|\mu\|_{C^{0,2}([0,T]\times B_R(\xi_0))}}.
  \end{align*}
  We note that $R \mapsto C_R$ is continuous, non-increasing in $R$ and $\lim_{R \downarrow 0} C_R = 1$. In contrast, $C_R$ does not necessarily converge to $1$ for $T \to 0$. Thus, the bounds become optimal locally in space but not locally in time.
  
  In order to derive the upper bound $Y(t,\xi) \le W(t,\xi,\xi_1)$ on $[0,\td T) \times B_r(\xi_1)$ we need $g(t,\xi)e^{\mu_t(\xi)} \le W(t,\xi,\xi_1)$ for a.a.\ $(t,\xi) \in [0,\td T) \times \partial B_r(\xi_1)$. For this it is sufficient to have
  \begin{equation*}
     W(t,\xi,\xi_1) = \td C |\xi-\xi_1|^\frac{2}{m-1}\left(F(\td T)-F(t)\right)^{\frac{-1}{m-1}} \ge H,
  \end{equation*}
  for a.a.\ $(t,\xi) \in [0,\td T) \times \partial B_r(\xi_1)$. This is satisfied if we choose $\td T = \td T_r$ by
  \begin{equation}\label{eqn:T(r)}
    \td T_r :=  F^{-1}\left(\frac{\td C^{m-1} r^2}{H^{m-1}}\right) = F^{-1}\left( r^2 \frac{C_{det}}{H^{m-1}}C_R\right).
  \end{equation}
  By Theorem \ref{thm:comp} and by continuity of $Y,W$ we conclude
    $$ 0 \le Y(t,\xi_1) \le W(t,\xi_1,\xi_1) = 0, \quad \forall t \in [0,\td T_r].$$

  Let $R_1 \in (0,R)$, $\xi_1 \in B_{R_1}(\xi_0)$ and $r = \dist(\xi_1,\partial B_R(\xi_0)) \ge R-R_1 > 0$. Resolving \eqref{eqn:T(r)} for $r$ yields
    $$ R(T) := R_1 = R - \sqrt{F(T)} \left(\frac{H}{\td C}\right)^\frac{m-1}{2}  = R - \sqrt{F(T)} \left(\frac{H^{m-1}}{C_{det}}\right)^\frac{1}{2} C_R^{-\frac{1}{2}}.    $$
  Hence,
    $$ Y(t,\xi) = 0, \quad \forall \xi \in B_{R(t)}(\xi_0),\ t \in [0,T_{stoch} \wedge T],$$
  where $T_{stoch} = T_R = F^{-1}\left( R^2 \frac{C_{det}}{H^{m-1}}C_R\right)$.
\end{proof}

\begin{remark}\label{rmk:noise_intens}
  Due to the explicit form of the estimates and the constant $C_R$ in Theorem \ref{thm:hole_filling_local_space}, the dependence of the bounds on the strength of the noise is obvious. In particular, when the noise intensity $\sum_{k=1}^N \|f_k\|_{C^2(B_R(\xi_0))}$ decreases to $0$, then the bounds from Theorem \ref{thm:hole_filling_local_space} approach the corresponding deterministic, optimal ones.
\end{remark}

We will now derive a second bound on the rate of collapse of balls for \eqref{eqn:SPME}. In contrast to Theorem \ref{thm:hole_filling_local_space} the construction of a suitable supersolution will be based on a temporal discretization, i.e.\ on freezing the noise at time $t=0$.

\begin{theorem}[Hole-filling theorem for small times]\label{thm:hole_filling_local_time}
  Let $\xi_0 \in \R^d$, $T,R  > 0$ and $X \in C((0,T] \times B_R(\xi_0))$ be an essentially bounded, non-negative, very weak subsolution to \eqref{eqn:SPME} with vanishing initial value $X_0$ on $B_R(\xi_0)$ and boundary value $g$ satisfying  $H := \|g\|_{L^\infty([0,T] \times \partial B_R(\xi_0))}  < \infty$.
  Define $T_{stoch}$ by
    $$ T_{stoch} := \sup\Bigg\{\td T \in [0,T] \Big|\  \td T C_{\td T} \le  R^2 \frac{C_{det}}{H^{m-1}}\Bigg\} ,$$
  where $t \mapsto C_t$ is a continuous, non-decreasing function with $\lim_{t \downarrow 0} C_t = 1$.

  Then $X_t$ vanishes in $B_{R_{stoch}(t)}(\xi_0)$ for all $t \in [0,T_{stoch}]$, where
     $$R_{stoch}(t) = R - \sqrt{t} \left(\frac{H^{m-1}}{C_{det}}\right)^\frac{1}{2} \sqrt{C_t}.$$
  Note that for $t \approx 0$ we recover the optimal rate from the deterministic case.
\end{theorem}
\begin{proof}
  The proof proceeds similarly to Theorem \ref{thm:hole_filling_local_space}. Hence, let $Y := e^\mu X$ be a very weak subsolution to \eqref{eqn:trans_inhomog} with initial value $Y_0 \equiv 0$ and boundary value $e^\mu g$.

  For $\xi_1 \in B_R(\xi_0)$, $\td T \in (0,T]$, $r \in (0,\dist(\xi_1,\partial B_R(\xi_0))]$ we again construct an explicit supersolution to \eqref{eqn:trans_inhomog} in $[0,\td T) \times B_r(\xi_1)$:
    $$W(t,\xi,\xi_1) := \td C e^{\mu_{0}(\xi)}|\xi-\xi_1|^\frac{2}{m-1}\left(\td T-t\right)^{\frac{-1}{m-1}}, \quad t \in [0,\td T), \xi \in B_r(\xi_1),$$
  where $\td C$ will be chosen below, depending on $\td T,R$ only. We compute:
  \begin{align*}
    \partial_t W(t,\xi,\xi_1) = \frac{1}{m-1} \td C e^{\mu_{0}(\xi)} |\xi-\xi_1|^\frac{2}{m-1} \left(\td T-t\right)^{\frac{-m}{m-1}}, 
  \end{align*}
  for all $(t,\xi) \in [0,\td T) \times B_r(\xi_1)$ and 
  \begin{align*}
    &\D \left(e^{-\mu_t(\xi)} W(t,\xi,\xi_1)\right)^m  \\
    &\hskip0pt \le \frac{\td C^m}{(m-1) C_{det}} \left(\td T-t\right)^\frac{-m}{m-1} |\xi-\xi_1|^\frac{2}{m-1} \Big( e^{m(\mu_{0}(\xi)-\mu_{t}(\xi))} \\
    &\hskip10pt + \frac{2(m-1)}{2  + d (m-1)} |\nabla e^{m(\mu_{0}(\xi)-\mu_{t}(\xi))}| R + (m-1)C_{det} R^2 |\D e^{m(\mu_{0}(\xi)-\mu_{t}(\xi))}| \Big) \\
    &\hskip0pt \le  \frac{\td C^m}{(m-1) C_{det}} \left(\td T-t\right)^\frac{-m}{m-1} |\xi-\xi_1|^\frac{2}{m-1} e^{m(\mu_{0}(\xi)-\mu_{t}(\xi))} \\
    &\hskip15pt \Big( 1 + C(d,m)R(1 + R) \|\mu_{0}-\mu_{\cdot}\|_{C^{0,2}( [0,\td T] \times B_R(\xi_0))} \Big)
  \end{align*}
  for all $(t,\xi) \in [0,\td T) \times B_r(\xi_1)$.
  We conclude that 
    $$  \partial_t W(t,\xi,\xi_1) \ge e^{\mu_t(\xi)} \D \left(e^{-\mu_t(\xi)} W(t,\xi,\xi_1)\right)^m $$
  on $[0,\td T) \times B_r(\xi_1)$ if
    $$ e^{(m-1)(\mu_{t}(\xi)-\mu_0(\xi))} \ge \frac{\td C^{m-1}}{C_{det}} \Big( 1 + C(d,m)R (1 + R) \|\mu_{0}-\mu_{\cdot}\|_{C^{0,2}([0,\td T] \times B_R(\xi_0))}  \Big),$$
  for all $(t,\xi) \in [0,\td T) \times B_r(\xi_1)$, which is satisfied for the choice
    $$\td C^{m-1} = \frac{C_{det}}{C_{\td T}}e^{(m-1)\|\mu_\cdot-\mu_0\|_{C^0([0,\td T] \times \partial B_R(\xi_0))}}$$
  with
  \begin{align*}
    C_{\td T} := \frac{ 1 + C(d,m)R (1 + R) \|\mu_{0}-\mu_{\cdot}\|_{C^{0,2}([0,\td T]\times B_R(\xi_0))} }{ e^{-2(m-1) \|\mu_{0}-\mu_\cdot\|_{C^0([0,\td T]\times B_R(\xi_0))}  }}.
  \end{align*}
    
  In order to derive the upper bound $Y(t,\xi) \le W(t,\xi,\xi_1)$ on $[0,\td T) \times B_r(\xi_1)$ we need $g(t,\xi) e^{\mu_t}(\xi) \le W(t,\xi,\xi_1)$ for a.e.\ $(t,\xi) \in [0,\td T) \times \partial B_r(\xi_1)$. For this to be true it is sufficient to have
  \begin{equation*}
     W(t,\xi,\xi_1) = \td C e^{\mu_{0}(\xi)}|\xi-\xi_1|^\frac{2}{m-1}\left(\td T-t\right)^{\frac{-1}{m-1}} \ge H e^{\mu_t (\xi)},
  \end{equation*}
  for a.a.\ $(t,\xi) \in [0,\td T) \times \partial B_r(\xi_1)$. This is satisfied if 
  \begin{equation*}
      \frac{\td T e^{(m-1)\|\mu_\cdot-\mu_0\|_{C^0([0,\td T] \times \partial B_R(\xi_0))}}}{\td C^{m-1}} \le  r^2 \frac{1}{H^{m-1}},
  \end{equation*}
  for which in turn it is sufficient to have
  \begin{equation}\label{eqn:T(r)_1}\begin{split}
      \td T C_{\td T} \le r^2 \frac{C_{det}}{H^{m-1}}.
  \end{split}\end{equation}
  Since the left hand side is continuous in $\td T$ we may choose $\td T$ as
  \begin{equation*}
     \td T := \sup\Big\{\td T \in [0,T] \Big|\ \td T C_{\td T} \le r^2\frac{C_{det}}{H^{m-1}} \Big\}.
  \end{equation*}

Note that $\td T \mapsto C_{\td T} > 0$ is continuous, non-decreasing and 
  $$ C_{\td T} \to 
    \begin{cases}
        1,& \quad  \text{ for } \td T \to 0 \\
        e^{2(m-1) \|\mu_{0}-\mu_\cdot\|_{C^0([0,\td T]\times \{\xi_0\})}},& \quad  \text{ for } R \to 0,
    \end{cases} $$
  i.e.\ we recover the optimal constant from the deterministic case for asymptotically small time, while locally in space the estimates will not be optimal.

  Let now $R_1 \in (0,R)$, $\xi_1 \in B_{R_1}(\xi_0)$ and $r = \dist(\xi_1,\partial B_R(\xi_0)) \ge R-R_1 > 0.$ 
  By Theorem \ref{thm:comp} and by continuity of $Y,W$ we conclude
    $$ 0 \le Y(t,\xi_1) \le W(t,\xi_1,\xi_1) = 0, \quad \forall t \in [0,\td T(r)].$$
  Resolving \eqref{eqn:T(r)_1} for $R_1$ yields
    $$ R(T) := R_1 = R - \sqrt{T} \left(\frac{H^{m-1}}{C_{det}}\right)^\frac{1}{2} \sqrt{C_T}.    $$
  Hence,
    $$ Y(t,\xi) = 0, \quad \forall \xi \in B_{R(t)}(\xi_0),\ t \in [0,T_{stoch}],$$
  where
  \begin{align*}
     T_{stoch} 
     &:= \td T(R) \\
     &:= \sup\Bigg\{\td T \in [0,T] \Big|\ \td T C_{\td T} \le R^2 \frac{C_{det}}{H^{m-1}} \Bigg\}.
  \end{align*}
\end{proof}

We are now ready to derive bounds on the speed of propagation for \eqref{eqn:SPME}. We give two formulations of this property 

\begin{theorem}[Finite speed of propagation]\label{thm:propagation}
  Let $X \in C((0,T]\times\mcO)$ be an essentially bounded, non-negative, very weak subsolution to the homogeneous Dirichlet problem to \eqref{eqn:SPME} and set $H := \|e^\mu X\|_{L^\infty(\mcO_T)}$. Then, for each $s \in [0,T]$ and every $h > 0$ there is a time-span $T_h > 0$ such that
   \begin{equation}\label{eqn:finite_prop_1}
       \supp (X_{s+t}) \subseteq B_h(supp (X_s)), \quad \forall t \in [0,T_h \wedge (T-s)].
   \end{equation}
   More precisely, $T_h$ is given by
    $$ T_h :=  F_h^{-1}\left(h^2 \frac{C_{det}}{H^{m-1}} C_h \right),$$
   where $ F_h(t) := \int_0^t e^{-(m-1)\inf_{\xi_0 \in \partial B_h(\supp(X_s))}\mu_r(\xi_0)}dr$ and $h \mapsto C_h$ is a continuous, non-increasing function satisfying $\lim_{h \downarrow 0} C_h = 1$. In particular,
     $$ \left| T_h - F_0^{-1}\left(h^2 \frac{C_{det}}{H^{m-1}}\right)\right| \to 0,\quad \text{for } h\to 0,$$
   with $ F_0(t) = \int_0^t e^{-(m-1)\inf_{\xi_0 \in \partial \overline{\supp(X_s)}}\mu_r(\xi_0)}dr $.
\end{theorem}
\begin{proof}
  Without loss of generality we assume $s = 0$. In order to avoid difficulties at the boundary we first replace $X$ by a solution to \eqref{eqn:SPME} on some large ball $B_R(0) \supseteq \mcO$, where we choose $R>0$ large enough, such that the boundary $\partial B_R(0)$ becomes ``invisible'' for the solution on $[0,T]$: Let $h >0$, $R>0$ such that 
    $$\td \mcO := B_R(0) \supseteq \bar B_{2h}(\supp(X_0)) \cup \mcO.$$
  By \cite[Theorem 1.3, Theorem 1.4, Theorem 1.12]{G11c} there is a unique, essentially bounded, non-negative, weak solution $X \in C((0,T]\times\td\mcO)$ to \eqref{eqn:SPME} on $\td\mcO$ with zero Dirichlet boundary conditions and initial condition $\td X_0 := X_0 \mathbbm{1}_{\mcO} \in L^\infty(\td\mcO)$. Since $\td X$ is a supersolution to the homogeneous Dirichlet problem to \eqref{eqn:SPME} on $\mcO$, by Theorem \ref{thm:comp} we have 
    $$X \le \td X, \quad \text{on } \mcO_T.$$
  Thus, it is sufficient to prove the claim for $\td X$. Hence, without loss of generality we may assume $X$ to be an essentially bounded, weak solution to \eqref{eqn:SPME} and $$\dist(\supp(X_0),\partial\mcO) > 2h.$$

  Let $\xi_0 \in \partial B_h(\supp (X_0)) \subseteq \mcO$. Then, $X_0 = 0$ on $B_h(\xi_0) \subseteq \mcO$ and Theorem \ref{thm:hole_filling_local_space} (with $R=h$) implies that $X_t$ vanishes on $B_{R_{stoch}(t)}(\xi_0)$ for all $t \in [0,T_{stoch}\wedge T]$, where $R_{stoch}(t)$ and $T_{stoch}$ given in Theorem \ref{thm:hole_filling_local_space} depend on $\xi_0$ via the constant $C_{h}$ and the function $F$. We note that $C_{h}$ may be uniformly estimated by
  \begin{align*}
    C_{h} \ge \bar C_h :=
    & \frac{e^{- h\|\mu\|_{C^{0,1}([0,T]\times (\mcO \cap \supp (X_0)^c))}}}{\Big(1+C(m)h(1+h)\|\mu\|_{C^{0,2}([0,T]\times (\mcO \cap \supp (X_0)^c))}\Big)^\frac{1}{m-1}}
  \end{align*}
  and $F$ by
    $$ \bar F_h(t) = \int_0^t e^{-(m-1)\inf_{\xi_0 \in \partial B_h(\supp (X_0))}\mu_r(\xi_0)}dr \ge F(t).$$
  Therefore, $T_{stoch}$ is uniformly bounded from below by 
    $$ \bar T_h :=  \bar F_h^{-1}\left(h^2 \frac{C_{det}}{H^{m-1}} \bar C_h \right),$$
  and $R_{stoch}(t)$ by
    $$\bar R_h(t) = h - \sqrt{\bar F_h(t)} \left(\frac{H^{m-1}}{C_{det}}\right)^\frac{1}{2} \bar C_h^{-\frac{1}{2}}.$$
  Hence, $X_t$ vanishes on $B_{\bar R_{h}(t)}(\partial B_h(\supp (X_0)))$ for all $t \in [0,\bar T_h \wedge T]$. 

  In particular, this implies that $X$ is a weak solution to the homogeneous Dirichlet problem to \eqref{eqn:SPME} on $[0,T_h \wedge T] \times \mcO \cap B_h(\supp(X_0))^c$. Since $X_0 \equiv 0$ on $\mcO \cap B_h(\supp(X_0))^c$ this implies $X_t \equiv 0$ on $\mcO \cap B_h(\supp(X_0))^c$ for all $t \in [0,T_h \wedge T]$.
\end{proof}

\begin{theorem}[Finite speed of propagation]\label{thm:propagation_time}
  Let $X \in C((0,T]\times\mcO)$ be an essentially bounded, non-negative, very weak subsolution to the homogeneous Dirichlet problem to \eqref{eqn:SPME} and set $H := \|X\|_{L^\infty(\mcO_T)}$. Then, for every $ s \in [0,T]$ 
      $$ supp (X_{s+t}) \subseteq  B_{\sqrt{t}\left(\frac{H^{m-1}}{C_{det}}\right)^\frac{1}{2}\sqrt{C_{t}}}(\supp(X_s)),\quad \forall t \in [0,T-s],$$
    where $t \mapsto C_{t}$ is a continuous, non-decreasing function with $C_{t} \to 1$ for $ t \to 0.$
\end{theorem}
\begin{proof}
  We argue as for Theorem \ref{thm:propagation} but apply Theorem \ref{thm:hole_filling_local_time} instead of Theorem \ref{thm:hole_filling_local_space}. Let $D := \diam(\mcO)$. We then estimate $C_{T}$ uniformly by
    $$C_{T} \le \bar C_{T} := \frac{ 1 + C(d,m)D (1 + D) \|\mu_{0}-\mu_{\cdot}\|_{C^{0,2}([0,T]\times (\mcO \cap \supp(X_0)^c))} }{ e^{-2(m-1)\|\mu_\cdot-\mu_0\|_{C^0([0,T] \times (\mcO \cap \supp(X_0)^c))}}}.$$
Hence, for 
  $$ \bar T(h) := \sup\Bigg\{\td T \in [0,T] \Big|\  \td T \bar C_{\td T} \le h^2 \frac{C_{det}}{H^{m-1}}\Bigg\} ,$$
we have $\bar T(h) \le T_{stoch}$ for all $\xi_0 \in \partial(B_h(\supp(X_0)))$ as in the proof of Theorem \ref{thm:propagation} and for
  $$\bar R(t) := h - \sqrt{t} \left(\frac{H^{m-1}}{C_{det}}\right)^\frac{1}{2} \sqrt{\bar C_{t}},$$
we have $\bar R(t) \le R_{stoch}(t)$ for all $t \in [0, T]$. In particular, for all $\xi_0 \in \partial(B_h(\supp(X_0)))$ we deduce
  $$ X_t(\xi_0) = 0, \quad  \forall t \in [0,\bar T(h)] $$
Arguing as for Theorem \ref{thm:propagation} this implies $\supp(X_t) \subseteq B_h(\supp(X_0))$ for all $t \le \bar T(h)$. Resolving for $h$ yields
  $$X_t \equiv 0 \text{ on } B_{\sqrt{t} \left(\frac{H^{m-1}}{C_{det}}\right)^\frac{1}{2} \sqrt{\bar C_{t}}}(\supp(X_0)),$$
\end{proof}

\begin{remark}[Unbounded domains $\mcO \subseteq \R^d$]\label{rmk:unbdd}
  In case of unbounded domains $\mcO\subseteq \R^d$ no pathwise uniqueness and existence theory (in the sense of existence of a stochastic flow) has been established for \eqref{eqn:SPME} so far. We note, however, that the simpler problem of constructing probabilistic solutions to \eqref{eqn:SPME} with $z^{(k)}$ being given as paths of Brownian motions has been solved in \cite{RRW07} for $d \ge 3$. 

  If the support of the initial condition $X_0 \in L^\infty(\mcO)$ is compact and bounded away from $\partial\mcO$ then the existence of corresponding essentially bounded, weak solutions $X$ to the homogeneous Cauchy-Dirichlet problem on short time intervals $[0,T]$ follows from the finite speed of propagation properties proved in this paper. The time of existence $T$ allowed by this approach is limited due to the support $\supp(X_t)$ reaching the boundary $\partial\mcO$. In particular, for the Cauchy problem no restriction on the time of existence has to be made. 

  For initial conditions $X_0$ with compact support, also uniqueness of essentially bounded, very weak solutions may be deduced from the methods of this paper at least on short time intervals $[0,T]$. Again, for the Cauchy problem no restriction on the time interval has to be supposed. 

  The case of initial conditions with unbounded support, however, remains open. 
\end{remark}

\section{Infinite dimensional random attractor}\label{sec:ra}

In this section we use the result of finite speed of propagation for SPME of the form \eqref{eqn:SPME} to prove that the random attractor associated to 
\begin{equation}\label{eqn:pert_SPME}\begin{split}
    d X_t         &= \D \left(|X_t|^m \sgn(X_t) \right)dt + \l X_t dt + \sum_{k=1}^N f_k X_t \circ dz^{(k)}_t, \text{ on } \mcO_T,  \\
    X(0)          &= X_0, \text{ on } \mcO, 
\end{split}\end{equation}
with homogeneous Dirichlet boundary conditions and $\l > 0$ has infinite fractal dimension. First, we will prove the existence of an RDS corresponding to \eqref{eqn:pert_SPME} in Proposition \ref{prop:generation_RDS}, then we will obtain the existence of an associated random attractor (Proposition \ref{prop:ex_ra}) and provide lower bounds on its Kolmogorov $\ve$-entropy (Theorem \ref{thm:inf_dim}).

In the following we assume the driving signals $z^{(k)}$ to be given as paths of a stochastic process with strictly stationary increments. More precisely, let $(\Omega,\mathcal{F},\mathcal{F}_t,\mathbb{P})$ be a filtered probability space, $(z_t)_{t \in \R}$ be an $\R^N$-valued adapted stochastic process and $((\O,\mcF,\P),(\theta_t)_{t \in \R})$ be a metric dynamical system. For notions and results from the theory of RDS and random attractors we refer to \cite[Section 1.2.1]{G11c}, \cite{A98,BL06,CDF97,CF94,S92}. We suppose:
\begin{enumerate}
  \item [$(S1)$] (Strictly stationary increments) For all $t,s \in \R$, $\o \in \O$:
                  $$z_t(\o)-z_s(\o) = z_{t-s}(\t_s \o),$$
                 where we assume $z_0 = 0$ for notational convenience only.
  \item [$(S2)$] (Regularity) $z_t$ has continuous paths.
  \item [$(S3)$] (Sublinear growth) $z_t(\o) = o(|t|)$ for $t \to -\infty$, for all $\o \in \O$.
\end{enumerate}

\subsection{Generation of an RDS and existence of a random attractor}\label{sec:gen_RDS}

If we set $f_{N+1} := \l$ and $z^{(N+1)}_t := t$, then \eqref{eqn:pert_SPME} is of the form \eqref{eqn:SPME} and Proposition \ref{prop:u_ex} implies the unique existence of a generalized weak solution $X(\cdot,s;\o)x$ with $X(s,s;\o)x = x$ for each $s \in \R$, $x \in L^1(\mcO)$ and driving signals $t \mapsto z_t^{(k)}(\o)$. Recall that $X(\cdot,s;\o)x$ is defined to be a solution to \eqref{eqn:pert_SPME} (resp.\ \eqref{eqn:SPME}) if 
\begin{equation}\label{eqn:inhomo_trans}
  Y(t,s;\o)(e^{\mu_s(\o)-\l s}x) := e^{\mu_t(\o)-\l t}X(t,s;\o)x, \quad  t \in [s,\infty),
\end{equation}
is a solution to \eqref{eqn:trans_inhomog} with initial condition $Y(s,s;\o)(e^{\mu_s(\o)-\l s}x) = e^{\mu_s(\o)-\l s}x$. We set
    \[ \vp(t-s,\t_s \o)x:= X(t,s;\o)x,\quad \text{for } t \ge s,\ \o \in \O,\ x \in L^1(\mcO)\]
and note that in \cite[Theorem 1.31]{G11c} strict stationarity of $z^{(k)}$ was only needed to prove the stochastic flow property for the solutions $X(t,s;\o)x$. Since the additional term $\l X_t dt$ in \eqref{eqn:pert_SPME} does not depend on time, the same proof as in \cite[Theorem 1.31]{G11c} still yields
\begin{proposition}\label{prop:generation_RDS}
  The map $\vp$ is a continuous RDS on $X=L^1(\mcO)$ and thus a quasi-weakly-continuous RDS on each $L^p(\mcO)$, $p \in [1,\infty)$. In addition, $\vp$ is a quasi-weakly$^*$-continuous RDS on $L^\infty(\mcO)$. $\vp$ satisfies comparison, i.e.\ for $x_1,x_2 \in X$ with $x_1 \le x_2$ a.e.\ in $\mcO$
    \[ \vp(t,\o)x_1 \le \vp(t,\o)x_2,\quad \text{a.e. in } \mcO.\]
  Moreover, $\vp$ satisfies $\vp(t,\o)0 = 0$ and
  \begin{enumerate}
   \item $x \mapsto \vp(t,\o)x$ is Lipschitz continuous on $X$, locally uniformly in $t$.
   \item $t \mapsto \vp(t,\o)x$ is continuous in $X$.
  \end{enumerate}
\end{proposition}

In the following let $\mcD$ be the universe of all random closed sets in $X$.

As pointed out above we may rewrite \eqref{eqn:pert_SPME} in the form of \eqref{eqn:trans_inhomog}. From \cite[Theorem 1.12, Theorem 1.31]{G11c} we deduce that there is a piecewisely smooth function $U(\o): (0,T] \to \R_+$ such that 
  $$\|\vp(t,\o)x\|_{L^\infty(\mcO)} \le U(t,\o), \quad \forall (t,\o) \in (0,T] \times \O.$$
Note that $U$ does not depend on the initial condition $x \in L^1(\mcO)$. This implies $\mcD$-bounded absorption for $\vp$ at time $t = 0$ with absorbing set being bounded with respect to the $\|\cdot\|_{L^\infty(\mcO)}$-norm. Moreover, for each $D \in \mcD$, $\vp(t,\o)D$ is locally equicontinuous in $(0,T] \times \mcO$, i.e.\ $\vp(t,\o)D = \{\vp(t,\o)x|\ x \in D\}$ is a set of equicontinuous functions on each compact set $K \subseteq (0,T] \times \mcO$. This yields $\mcD$-asymptotic compactness for $\vp$ as in \cite[Lemma 3.2]{G11c}. We conclude:

\begin{proposition}[Existence of a random attractor]\label{prop:ex_ra}
   The RDS $\vp$ has a $\mcD$-random attractor $\mcA$ (as an RDS on $L^1(\mcO)$). $\mcA$ is compact in each $L^p(\mcO)$ and attracts all sets in $\mcD$ in $L^p$-norm, $p \in [1,\infty)$.

  Moreover, $\mcA(\o)$ is a bounded set in $L^\infty(\mcO)$ and the functions in $\mcA(\o)$ are equicontinuous on every compact set $K \subseteq \mcO$.
\end{proposition}

\subsection{Lower bounds on the Kolmogorov \texorpdfstring{$\ve$}{eps}-entropy} We will now prove that the random attractor constructed in Proposition \ref{prop:ex_ra} has infinite fractal dimension in $L^1(\mcO)$. 

A precompact set $\mcA \subseteq X$ can be covered by a finite number of balls of radius $\ve$ for each $\ve > 0$. Let $N_\ve(\mcA)$ be the minimal number of such balls. Then, the Kolmogorov $\ve$-entropy of $\mcA$ is defined by
  $$\mathbbm{H}_\ve(\mcA) := \log_2(N_\ve(\mcA)).$$
The fractal dimension of $\mcA$ is defined by
  $$ d_f(\mcA) = \limsup_{\ve \to 0} \frac{\mathbbm{H}_\ve(\mcA)}{\log_2(\frac{1}{\ve})}. $$
We obtain

\begin{theorem}[Lower bounds on the Kolmogorov $\ve$-entropy]\label{thm:inf_dim}
  Let $\mcA$ be the random attractor for $\vp$ constructed in Proposition \ref{prop:ex_ra}. Then, the Kolmogorov $\ve$-entropy of $\mcA$ is bounded below by 
    $$ \mathbbm{H}_\d (\mcA(\o)) \ge C(\o) \d^{\frac{-d(m-1)}{2+d(m-1)}},\quad \forall \o \in \O,$$
  where $C(\o) > 0$ is a constant which may depend on $m, d$. In particular, the fractal dimension $d_f(\mcA(\o))$ is infinite for all $\o \in \O$.
\end{theorem}
\begin{proof}
  The proof is inspired by \cite[Theorem 4.1]{EZ08} and \cite[Theorem 3.3]{G11b}. In order to prove the lower bound on the Kolmogorov $\ve$-entropy we consider the unstable manifold of the equilibrium point $0$ defined by
  \begin{align*}
    \mcM^+(0,\o) := \{& u_0 \in X\ |\ \exists \text{ function } u: (-\infty,0] \to X, \text{ such that}\\ &\vp(t;\t_{-t}\o)u(-t) = u_0
                                    \text{ for all } t \ge 0 \text{ and } \|u(t)\|_X \to 0 \text{ for } t \to -\infty \}.
  \end{align*}
  Since $\mcA(\o)$ attracts all deterministic sets we have
    $$ \mcM^+(0,\o) \subseteq \mcA(\o),\quad \forall \o \in \O.$$
  Therefore, it is sufficient to derive a lower bound on the Kolmogorov $\ve$-entropy for the unstable manifold of $0$. 

  In order to construct an element $u_0 \in \mcM^+(0,\o)$ we need to find a function $u: (-\infty,0] \to X$ converging to $0$ for $t \to -\infty$ such that 
    $$u_0 = \vp(t;\t_{-t}\o)u(-t) = X(0,-t;\o)u(-t) = Y(0,-t;\o)\left(e^{\mu_{-t}(\o)+\l t} u(-t) \right),\ \forall t \ge 0,$$
  where we used \eqref{eqn:inhomo_trans}. By defining $u(-t) := e^{-\mu_{-t}(\o)-\l t}v(-t)$, due to (S3) it is enough to find a bounded function $v: (-\infty,0] \to X$ such that 
  \begin{equation}\label{eqn:v}
    u_0 = Y(0,-t;\o)v(-t),\quad \forall t \ge 0.
  \end{equation}

  We note that \eqref{eqn:trans_inhomog} in case of \eqref{eqn:pert_SPME} reads
  \begin{equation}\begin{split}\label{eqn:trans_gen}
      \partial_t Y(t,s;\o)x &= e^{\mu_t(\o)-\l t} \D \Phi(e^{-\mu_t(\o)+\l t} Y(t,s;\o)x),\\
      Y(s,s;\o)x &= x,
  \end{split}\end{equation}
  for a.e.\ $t \ge s$. For $x \in L^\infty(\mcO)$ let $Y(t,s;\o)x \in C((0,\infty)\times\mcO)$ denote the corresponding essentially bounded, weak solution to \eqref{eqn:trans_gen} given by Proposition \ref{prop:gen_ex}.

  In order to find a function $v$ satisfying \eqref{eqn:v}, we use a time scaling to transform \eqref{eqn:trans_gen} from the infinite time interval $(-\infty,0]$ into a PDE on a finite time interval. Let $\d > 0$ small enough such that $(m-1)\l -\d > 0$ and set $\eta := \frac{(m-1)\l-\d}{m+1}$. Then \eqref{eqn:trans_gen} may be rewritten as
  \begin{equation*}
     \partial_t Y(t,s;\o)x = e^{\d t} e^{\mu_t+\mu t} \D \Phi(e^{-\mu_t+\mu t} Y(t,s;\o)x).
  \end{equation*}
  We define $T = \frac{1}{\d}$ and
  \begin{align*}
     F(t) &:= \frac{e^{\d t}}{\d}: (-\infty,0] \mapsto (0,T],\\     
    G(t) &= F^{-1}(t) = \frac{\log(\d t)}{\d}: (0,T] \mapsto (-\infty,0].
  \end{align*}

  We note $G \in C^1(0,T]$ with $G'(t) > 0$, $G(T) = 0$ and $G(t) \to -\infty$ for $t \to 0$. Let $U(t,s;\o)x := Y(G(t),G(s);\o)x$ for $t \ge s$, $t,s \in (0,T]$. Then $U(\cdot,s;\o)x$ is a weak solution to 
  \begin{equation}\begin{split}\label{eqn:det_PME}
    \partial_t U(t,s;\o)x 
    &= e^{\mu_{G(t)}+\eta {G(t)}} \D \Phi(e^{-\mu_{G(t)}+\eta {G(t)}} U(t,s;\o)x), \text{ on } [s,\infty) \times \mcO \\
    U(s,s;\o)x &= x.
  \end{split}\end{equation}
  The rigorous proof of this transformation proceeds by considering a non-degenerate approximation $\Phi^{(\d)}(r) := \Phi(r) + \d r$ and smoothed coefficients $\mu^{(\d)}$. In this case the transformation is a direct consequence of the classical chain-rule. One may then use local equicontinuity and uniform boundedness of the approximating solutions $Y^{(\d)}$ to pass to the limit.

  Thus, we can solve \eqref{eqn:det_PME} on each interval $[\tau,T]$ with $\tau > 0$. In order to construct the required function $v: (-\infty,0] \to X$ we aim to solve \eqref{eqn:det_PME}  on the whole interval $[0,T]$. Let $\rho_1(t) := e^{\mu_{G(t)}+\eta {G(t)}}$, $\rho_2(t) := e^{-\mu_{G(t)}+\eta {G(t)}}$. Due to condition (S3), for each $\ve > 0$ there is a $t_0(\ve) < 0$ small enough, such that 
    $$\|\mu_{G(t)}\|_{C^n(\mcO)} \le \ve \left(\sum_{k=1}^N \|f_k\|_{C^n(\mcO)}\right) |G(t)|, \quad \forall t \le t_0(\ve),\ n \in \N.$$
  Choosing $\ve >0$ small enough we thus obtain
    $$\|\rho_1(t)\|_{C^n(\mcO)} \le e^{\|\mu_{G(t)}\|_{C^0(\mcO)}+\eta {G(t)}} P(\|\mu_{G(t)}+\eta {G(t)}\|_{C^n(\mcO)})  \to 0, \quad \text{for } t \to 0,$$
  for some polynomial $P$. Similarly,
  \begin{align*}
     \frac{|\partial_{\xi_{i_1},...,\xi_{i_n}} \rho_1(t)|^2}{\rho_1(t)} 
     &\le \frac{\rho_1(t)^2 P(\|\mu_{G(t)}+\eta {G(t)}\|_{C^n(\mcO)}) }{\rho_1(t)} \\
     &\le e^{\|\mu_{G(t)}\|_{C^0(\mcO)}+\eta {G(t)}} P(\|\mu_{G(t)}+\eta {G(t)}\|_{C^n(\mcO)})
     \to 0,\quad \text{for } t \to 0,
  \end{align*}
  for all $i_1,...,i_n \in \{1,...,d\}$. The same reasoning applies for $\rho_2$. In particular, $\rho_1,\rho_2 \in C^{0,n}(\bar\mcO_T)$ for all $n \in \N$. Hence, \eqref{eqn:det_PME} is of the form \eqref{eqn:general_trans} and Proposition \ref{prop:gen_ex} implies the existence of a very weak solution 
    $$U(\cdot,0;\o)x \in L^\infty([0,T]\times\mcO) \cap C((0,T]\times\mcO)$$
  with homogeneous Dirichlet boundary conditions for each initial condition $x \in L^\infty(\mcO)$. 

  Reversing the time transformation we define
    $$v(t) := U(F(t),0;\o)x, \quad t \in (-\infty,0].$$
  Uniqueness of essentially bounded, very weak solutions to \eqref{eqn:det_PME} (Theorem \ref{thm:gen_comp}) implies 
    $$ U(t,s;\o)x = U(t,r;\o)U(r,s;\o)x, \quad \forall 0 \le s \le r \le t \le T. $$
  Hence,
  \begin{align*}
    v(0) 
    &= U(F(0),0;\o)x 
    = U(F(0),F(s);\o)U(F(s),0;\o)x \\
    &=  U(F(0),F(s);\o)v(s) 
    = Y(0,s;\o)v(s), 
  \end{align*}
  for all $s < 0$. Consequently, $v(0) \in \mcM^+(0,\o)$ for each $x \in L^\infty(\mcO)$. 

  In order to use this construction of elements $v(0) \in \mcM^+(0,\o)$ to derive a lower bound on the Kolmogorov $\ve$-entropy of $\mcM^+(0,\o)$ we consider solutions to \eqref{eqn:det_PME} so that the final values $v(0)=U(F(0),0;\o)x$ are sufficiently far apart (w.r.t.\ the $L^1$-norm): For $\ve > 0$ small enough we can find a finite set $R_\ve = \{\xi_i\} \subseteq \mcO$ such that
  \begin{align*}
     B(\ve,\xi_i) \cap B(\ve,\xi_j)     &= \emptyset, \quad \text{for } i \ne j, \\
     |R_\ve|                            &\ge C \ve^{-d},\\
     \bar B(\ve,\xi_i)                       &\subset \mcO, \quad \forall i.
  \end{align*}

  Let $x^i_0 := M\mathbbm{1}_{B(\frac{\ve}{2},\xi_i)}$ and $M = (m \ve)^{\frac{2}{m-1}}$, where $m >0$ will be specified below. By Proposition \ref{prop:gen_ex}, 
    $$H^i := \|U^i(\cdot,0;\o)x\|_{L^\infty(\mcO_T)} \le C \|x^i_0\|_{L^\infty(\mcO)} \le C (m\ve)^{\frac{2}{m-1}}.$$
  Thus, the bound on the rate of expansion of the support of $U^i$ given in Theorem \ref{thm:gen_expansion} becomes
      $$ \supp (U^i_{t}) \subseteq  B_{C \ve m \sqrt{C_t} \sqrt{t}}(\supp(x_0^i)) \subseteq B_{C \ve m \sqrt{C_T}\sqrt{T} + \frac{\ve}{2}}(\xi_i),\quad \forall t \in [0,T],$$
  where $t \mapsto C_{t}$ is a continuous function. Thus, choosing $m$ small enough yields 
      $$ \supp (U^i_{t}) \subseteq  B_{\ve}(\xi_i),\quad \forall t \in [0,T].$$
  Hence, $U^i$, $U^j$ have disjoint support on $[0,T]$. Therefore, also 
    $$ U^{m}(t,\xi) = \sum_{i=1}^{|R_\ve|} m_i U^i(t,\xi), $$
  for each $m \in \{0,1\}^{|R_\ve|}$ is a very weak solution to \eqref{eqn:det_PME} with homogeneous Dirichlet boundary conditions. For $m^1 \ne m^2$ let $i$ such that $m^1_i \ne m^2_i$. By  Proposition \ref{prop:lower_l1} we observe
    $$ \|U^{m^1}(T) - U^{m^2}(T)\|_{L^1(\mcO)} \ge \|U^i(T)\|_{L^1(\mcO)} \ge e^{-CT} \|U^i(0)\|_{L^1(\mcO)} \ge C e^{-CT} \ve^{\frac{2}{m-1} + d}.$$
  Hence,
    $$ \mathbbm{H}_\d (\mcA(\o)) \ge \mathbbm{H}_\d (\mcM^+(0,\o)) \ge \log_2 2^{|R_\d|} \ge C(\o) \d^{-\frac{d(m-1)}{2+d(m-1)}}$$
  and   
    $$ d_f(\mcA(\o)) \ge d_f(\mcM^+(0,\o)) = \limsup_{\d \to 0} \frac{\mathbbm{H}_\d (\mcM^+(0,\o))}{log_2(\frac{1}{\d})} = \infty.$$
\end{proof}

\appendix

\section{Finite speed of propagation for more general perturbations}

In Section \ref{sec:finite_speed} we proved finite speed of propagation for \eqref{eqn:SPME} via the transformed equation \eqref{eqn:trans_inhomog}. The precise structure of the spatially dependent perturbing factors $e^{\mu},e^{-\mu}$ has been used to provide explicit and locally optimal bounds on the rate of hole-filling. By disregarding the optimality of the estimates, more general perturbations may be allowed. Such an extension of the results of Section \ref{sec:finite_speed} is required in Section \ref{sec:ra} in order to prove lower bounds for the Kolmogorov $\ve$-entropy of the random attractor. In this section we provide some details on the proof of finite speed of propagation for more general perturbing factors. We consider the homogeneous Dirichlet problem for
\begin{equation}\label{eqn:general_trans}\begin{split}
  \partial_t Y_t         &= \rho_1 \D \Phi(\rho_2 Y_t), \text{ on } \mcO_T \\
   Y(0)                   &= Y_0, \text{ on } \mcO,
\end{split}\end{equation}
where $\rho_1,\rho_2 \in C^{0,2}(\bar\mcO_T)$ are non-negative. (Local, generalized, very) weak solutions to \eqref{eqn:general_trans} are defined analogously to Definition \ref{def:weak_soln}. In particular, a function $Y \in L^1(\mcO_T)$ with $\Phi(\rho_2 Y) \in L^1(\mcO_T)$ satisfying
   \begin{equation}\label{eqn:very_weak_rough_transformed}\begin{split}
        \int_{\mcO_T} Y \partial_r \eta\ d\xi dr + \int_\mcO Y_0 \eta_0\ d\xi 
        \ge &-\int_{\mcO_T} \Phi(\rho_2 Y) \D (\rho_1 \eta )\ d\xi dr \\
        &+ \int_{\Sigma_T} \Phi(\rho_2 g) \partial_\nu (\rho_1 \eta) d\vartheta dr, 
   \end{split}\end{equation}            
  for all non-negative $\eta \in C^{1,2}(\bar\mcO_T)$ with $\eta_{|\mcP_T} = 0$ and for some functions $Y_0 \in L^1(\mcO)$, $\Phi(g) \in L^1(\Sig_T)$ is said to be a very weak subsolution to the (inhomogeneous) Cauchy-Dirichlet problem to \eqref{eqn:general_trans}.

\subsection{Existence of very weak solutions to \texorpdfstring{\eqref{eqn:general_trans}}{(\ref{eqn:general_trans})}} Let $Y_0\in L^\infty(\mcO)$. We will only require the existence of solutions to \eqref{eqn:general_trans} with homogeneous Dirichlet boundary conditions (i.e.\ $g \equiv 0$) and for $\rho_1,\rho_2$ satisfying one of the following conditions
\begin{enumerate}
    \item[(A1)] $\rho_2$ is strictly positive on $[0,T] \times \bar\mcO$,
    \item[(A2)] $\rho_2$ is strictly positive on $(0,T] \times \bar\mcO$ and $\|\rho_2(t)\|_{C^{2}(\mcO)} \to 0 \text{ for  } t \to 0.$
\end{enumerate}

The construction of solutions for \eqref{eqn:general_trans} relies on a smooth, non-degenerate approximation of $\Phi$. 
 I.e.\ for $\d > 0$ let $$\Phi^{(\d)}(r) := \Phi(r) + \d r,$$
$\rho_1^{(\d)},\rho_2^{(\d)} \in C^\infty(\mcO_T)$ be approximations of $\rho_1,\rho_2$ in $C^{0,2}(\bar\mcO_T)$ and let $Y_0^{(\d)} \in C^\infty(\mcO)$ be smooth approximations of $Y_0$ in $L^\infty(\mcO)$. We consider the approximating problems 
\begin{equation}\label{eqn:approx}\begin{split}
  \partial_t Y^{(\d)}_t        &= \rho_1^{(\d)} \D \left( \Phi(\rho_2^{(\d)}) \Phi^{(\d)}(Y^{(\d)}_t) \right), \text{ on } \mcO_T \\
  Y^{(\d)}(0)                  &= Y_0^{(\d)}, \text{ on } \mcO ,
\end{split}\end{equation}
with homogeneous Dirichlet boundary conditions. Since \eqref{eqn:approx} is a non-degenerate, quasilinear PDE with smooth coefficients, standard results imply the unique existence of a classical solution $Y^{(\d)}$ (cf.\ e.g.\ \cite{LSU67}).

The main ingredient of the construction of solutions to \eqref{eqn:general_trans} is the following a-priori $L^\infty$ bound
\begin{lemma}\label{lemma:gen_l_infty_bound}
   Let $M := \|Y_0\|_{L^\infty(\mcO)} < \infty$ and assume (A1) or (A2). Then, there are constants $C, \d_0=\d_0(M) > 0$ such that 
     $$\sup_{\d \in [0,\d_0]} \|Y^{(\d)}\|_{C^0([0,T]\times\overline{\mcO})} \le C \|Y_0\|_{L^\infty(\mcO)} < \infty.$$
\end{lemma}
\begin{proof}
  {\it Case (A1):} The proof relies on a combination of explicit supersolutions to \eqref{eqn:PME} with an interval splitting technique as it has been used in \cite{BR11,G11c}.

  In the following let $\vp \in C^2(\mcO)$ be the solution to
  \begin{align*}
    \D \vp &= -1, \quad \text{ on } \mcO \\
    \vp & = 1, \quad \text{ on } \partial \mcO.
  \end{align*}
  By the maximum principle we have $\vp \ge 1$.

  Since $\{\rho_2^{(\d)}\}_{\d \in [0,1]}$ is a compact set in $C^{0,2}(\mcO_T)$ and may be chosen such that
    $$\inf_{\d \in [0,1],\ (t,\xi) \in [0,T] \times \mcO} \rho^{(\d)}_2(t,\xi) > 0,$$
  we have
   $$\eta_i^{(\d)} := \Phi\left(\frac{\rho^{(\d)}_2}{\rho^{(\d)}_2(\tau_i)}\right) \in C^{0,2}(\mcO_T)$$
  with $\eta_i^{(\d)}(t) \to 1$ in $C^{2}(\mcO)$ for $t \to \tau_i$ uniformly in $\d \in [0,1]$ and $\tau_i \in [0,T]$. Hence,
    $$\D(\vp \eta^{(\d)}_i) = -\eta^{(\d)}_i + 2 \nabla \vp \cdot \nabla \eta^{(\d)}_i + \vp \D \eta^{(\d)}_i \le -\frac{1}{2}, \quad \forall \xi \in\mcO,\ \d \in [0,1]$$
  and all $|t-\tau_i|$ small enough. We can thus choose a finite partition $0 = \tau_0 < \tau_1 < \tau_2 < ... < \tau_N = T$ of $[0,T]$ such that
    $$\sup_{\d \in [0,1]}\D \left( \vp \Phi\left( \frac{\rho_2^{(\d)}}{\rho_2^{(\d)}({\tau_i})}\right)\right)  \le -\frac{1}{2},\quad \text{on } [\tau_i,\tau_{i+1}] \times \mcO,$$
  for all $i = 0,...,N-1$.      

  We will prove the bound iteratively over $i=0,...,N-1$. Suppose the bound has been shown on $[0,\tau_i]$ for some $i \ge 0$ and let $\|Y_{\tau_i}\|_{L^\infty(\mcO)} \le C_i M$. Choosing 
    $$K^{(i)}(t,\xi) := \vp(\xi)^\frac{1}{m}\frac{\|\rho_2^{(\d)}\|_{C^{0,2}(\mcO_T)}}{\rho_2^{(\d)}(\tau_i,\xi)} C_i M \in C^{2}(\mcO_T),$$
  we have $K^{(i)}(\tau_i,\xi) \ge \|Y_{\tau_i}\|_{L^\infty(\mcO)}$, $\partial_t K^{(i)} = 0$ and
  \begin{align*}
    &\rho_1^{(\d)} \D \left( \Phi(\rho_2^{(\d)}) \Phi^{(\d)}( K^{(i)}) \right) \\
    &= \rho_1^{(\d)} \D \left( \Phi(\rho_2^{(\d)} K^{(i)}) \right) + \d \rho_1^{(\d)} \D \left( \Phi(\rho_2^{(\d)}) K^{(i)} \right)\\
    &\le \|\rho_2^{(\d)}\|_{C^{0,2}(\mcO_T)} C_i M \rho_1^{(\d)}\left(- \frac{1}{2} \|\rho_2^{(\d)}\|_{C^{0,2}(\mcO_T)}^{m-1}(C_iM)^{m-1}  + \d \D \left(\frac{\Phi(\rho_2^{(\d)})\vp^\frac{1}{m}}{\rho_2^{(\d)}({\tau_i})} \right)\right) \\
    &\le 0,
  \end{align*}
  by the choice of the partition $\{\tau_i\}_{i=0,...,N}$, for all $\d \le \d_0(M)$ small enough. 

  Consequently, $K^{(i)}$ is a supersolution to \eqref{eqn:approx} on $[\tau_i,\tau_{i+1}]\times\mcO$ and the upper bound follows since $K^{(i)}(t,\xi) \le C_{i+1}M$, with $C_{i+1}$ depending on the data only. The derivation of the lower bound proceeds analogously. 

 {\it Case (A2):} We only need to prove the claim on some small interval $[0,\tau_1]$ with $\tau_1 > 0$, since case (A1) may be applied on $[\tau_1,T]$ subsequently. Choose $\tau_1 \in (0,T]$ such that 
        $$\sup_{\d \in [0,1]}\D \left( \vp \Phi\left( \rho_2^{(\d)} \right) \right) \le -\frac{1}{2},\quad \text{on } [0,\tau_{1}] \times \mcO.$$
  This is possible since $\|\rho_2(t)\|_{C^{2}(\mcO)} \to 0 \text{ for  } t \to 0$ by assumption. Let $K^{(0)}(t,\xi) := \vp(\xi)^\frac{1}{m} M$. Then $\partial_t K = 0$ and
  \begin{align*}
    \rho_1^{(\d)} \D \left( \Phi(\rho_2^{(\d)}) \Phi^{(\d)}( K) \right) 
    &= \rho_1^{(\d)} \D \left( \Phi(\rho_2^{(\d)} K) \right) + \d \rho_1^{(\d)} \D \left( \Phi(\rho_2^{(\d)}) K \right)\\
    &= M \rho_1^{(\d)} \left(M^{m-1} \D \left(\vp\Phi(\rho_2^{(\d)}) \right) + \d \D \left(\Phi(\rho_2^{(\d)} ) \vp^\frac{1}{m}\right) \right) 
    \le 0,
  \end{align*}
  on $[0,\tau_1]\times\mcO$ for $\d \le \d_0(M)$ small enough. Hence, $K^{(0)}$ is a supersolution to \eqref{eqn:approx} on $[0,\tau_1] \times \mcO$ and 
    $$Y^{(\d)} \le K^{(0)} \le C M, \quad \text{ on } [0,\tau_1] \times \mcO.$$
  The lower bound may be derived analogously.
\end{proof}

\begin{proposition}[Existence of very weak solutions to \eqref{eqn:general_trans}]\label{prop:gen_ex}
   Let $Y_0 \in L^\infty(\mcO)$ and assume (A1) or (A2). Then, there exists a very weak solution $Y \in C((0,T] \times \mcO)$ to \eqref{eqn:general_trans} with Dirichlet boundary conditions satisfying 
    $$ Y_{L^\infty(\mcO_T)} \le C\|Y_0\|_{L^\infty(\mcO)}, $$
  for some constant $C > 0$.
\end{proposition}
\begin{proof}
  Based on the uniform $L^\infty$ estimate for the approximating solutions $Y^{(\d)}$ derived in Lemma \ref{lemma:gen_l_infty_bound}, we obtain local equicontinuity of $Y^{(\d)}$ in $\mcO$ by \cite{DB83} (cf.\ also \cite[Theorem 1.12]{G11c}. I.e.\ $Y^{(\d)} \in C(K)$ for each compact set $K \subseteq (0,T] \times \mcO$ with modulus of continuity independent of $\d > 0$.
  
  By a diagonal argument it follows that there exists a $Y \in C((0,T]\times\mcO)$ with $\|Y\|_{L^\infty(\mcO_T)} \le C\|Y_0\|_{L^\infty(\mcO)}$ such that $Y^{\d} \to Y$ (passing to a subsequence if necessary) locally uniformly on $\mcO$. By dominated convergence, this implies that $Y$ is a very weak solution to \eqref{eqn:general_trans}. 
\end{proof}

\subsection{Comparison and uniqueness for \texorpdfstring{\eqref{eqn:general_trans}}{(\ref{eqn:general_trans})}}\label{sec:gen_comp}

We now prove a comparison result for \eqref{eqn:general_trans}. In particular, this implies Theorem \ref{thm:comp} since sub/supersolutions to \eqref{eqn:SPME} are defined in terms of solutions to \eqref{eqn:trans_inhomog} and thus it is enough to prove the comparison result for \eqref{eqn:trans_inhomog}. We will assume either of 
\begin{enumerate}
   \item[(A1')] $\rho_1$ is strictly positive on $[0,T] \times \bar\mcO$,
   \item[(A2')] $\rho_1$ is strictly positive on $(0,T] \times \bar\mcO$ and 
      $$\left\|\frac{|\nabla\rho_1(t)|^2}{\rho_1(t)}\right\|_{C^0(\mcO)}+\left\|\frac{|\D\rho_1(t)|^2}{\rho_1(t)}\right\|_{C^0(\mcO)} \to 0,\ \text{for } t \to 0.$$
  \end{enumerate}

\begin{theorem}[Comparison for very weak solutions]\label{thm:gen_comp}
  Let $Y^{(1)},Y^{(2)}$ be essentially bounded sub/supersolutions to \eqref{eqn:general_trans} with initial conditions $Y^{(1)}_0 \le Y^{(2)}_0$ and boundary data $g^{(1)} \le g^{(2)}$ a.e.\ in $\mcO$ respectively. Assume either (A1') or (A2'). Then,
    $$Y^{(1)} \le Y^{(2)},\quad\text{a.e.\ in } \mcO.$$
  In particular, essentially bounded, very weak solutions are unique.
\end{theorem}
\begin{proof}
   The proof proceeds similar to \cite[Theorem 1.3]{G11c}. Let $Y^{(1)},\ Y^{(2)}$ be as in the statement, $Y := Y^{(1)}-Y^{(2)}$ and $g := g^{(1)}-g^{(2)}$. Then
  \begin{align*}
      &\int_{\mcO_T} Y \partial_r\eta \ d\xi dr \\
      &\ge -\int_\mcO (Y_0^{(1)} - Y_0^{(2)})\eta_0 d\xi - \int_{\mcO_T} \left( \Phi(\rho_2 Y^{(1)})-\Phi(\rho_2 Y^{(2)}) \right) \D(\rho_1 \eta ) \ d\xi dr \\
      &\hskip10pt + \int_{\Sigma_T} (\Phi(\rho_2 g^{(1)} )-\Phi(\rho_2 g^{(2)})) \partial_\nu (\rho_1 \eta) d\vartheta dr \\
      &\ge -\int_\mcO Y_0 \eta_0 d\xi - \int_{\mcO_T} a Y \D(\rho_1 \eta ) \ d\xi dr  + \int_{\Sigma_T} (\Phi(\rho_2 g^{(1)} )-\Phi(\rho_2 g^{(2)})) \partial_\nu (\rho_1 \eta) d\vartheta dr ,
  \end{align*}
  for all non-negative $\eta \in C^{1,2}(\bar\mcO_T)$ with $\eta = 0$ on $\mcP_T$, where
  $$ a_t := 
  \begin{cases}
      \frac{\Phi(\rho_2(t) Y^{(1)}_t)-\Phi(\rho_2(t) Y^{(2)}_t)}{Y^{(1)}_t-Y^{(2)}_t} &, \text{ for }  Y^{(1)}_t \ne Y^{(2)}_t \\
      0         &, \text{ otherwise}.
  \end{cases} $$

  {\it Case (A1'):} Let $\rho_1^{(\ve)} \in C^\infty(\mcO_T)$ be a smooth approximation of $\rho_1$ in $C^{0,2}(\mcO_T)$, such that $\|\rho^{(\ve)}_1-\rho_1\|_{C^{0,2}(\mcO_T)} \le \ve^2$. By equicontinuity of $t \mapsto \rho_1^{(\ve)}(t)$ in $C^{2}(\mcO)$ we can choose a partition $0 = \tau_0 < ... < \tau_N = T$ such that
  \begin{equation}\label{eqn:choice_tau_2}\begin{split}
     &C_1 \|\rho_{1}^{(\ve)}(\tau_i)\|_{C^0(\mcO)} \left(\left\|\nabla \left(\frac{\rho_1^{(\ve)}}{\rho_{1}^{(\ve)}(\tau_i)}\right) \right\|_{C^{0}([\tau_i,\tau_{i+1}] \times \mcO)}^2 + \left\|\D \left(\frac{\rho_1^{(\ve)}}{\rho_{1}^{(\ve)}(\tau_i)}\right) \right\|_{C^{0}([\tau_i,\tau_{i+1}] \times \mcO)}^2\right) \\ &\le \frac{c}{4}, \quad \forall i = 0,...,N-1,\ \ve > 0,
  \end{split}\end{equation}
  where $c,C_1 > 0$ are constants that will be specified below (depending on $\|a\|_{L^\infty(\mcO_T)}$ only). Let $\g := \max_{i = 0,...,N-1}|\tau_{i+1}-\tau_i|$.

  We prove $Y \le 0$ a.e.\ via induction over $i=0,...,N-1$. Thus, assume $Y \le 0$ on ${[0,\tau_i] \times \mcO}$ almost everywhere. We can modify $\tau_i$ so that \eqref{eqn:choice_tau_2} is preserved and $Y(\tau_i) \le 0$ a.e.\ in $\mcO$. Define $\mcO_i := [\tau_i,\tau_{i+1}] \times \mcO$, $\Sig_i = [\tau_i,\tau_{i+1}] \times \partial\mcO$, $\mcP_i = \Sig_i \cup (\{T\} \times \mcO)$. Then
  \begin{align*}
    \int_{\mcO_i} Y \big( \partial_r \eta + a \D(\rho_1 \eta) \big) \ d\xi dr 
    \ge &-\int_\mcO Y_{\tau_i}\eta_{\tau_i} d\xi \\
      &+ \int_{\Sigma_i} (\Phi(\rho_2 g^{(1)} )-\Phi(\rho_2 g^{(2)})) \partial_\nu (\rho_1 \eta) d\vartheta dr,
  \end{align*} 
  for all non-negative $\eta \in C^{1,2}([\tau_i,\tau_{i+1}] \times \bar\mcO)$ with $\eta = 0$ on $\mcP_i$. Since $\eta \ge 0$ on $\mcO_i$, we have $\partial_\nu (\rho_1 \eta) \le 0$ on $\Sigma_i$ and thus
    $$ -\int_\mcO Y_{\tau_i} \eta_{\tau_i} d\xi + \int_{\Sigma_i} (\Phi(\rho_2 g^{(1)} )-\Phi(\rho_2 g^{(2)})) \partial_\nu (\rho_1\eta) d\vartheta dr \ge 0.$$
  We conclude,
  \begin{align*}
    \int_{\mcO_i} Y \big( \partial_r \eta + a \D(\rho_1 \eta) \big) \ d\xi dr \ge 0,
  \end{align*} 
  for all non-negative $\eta \in C^{1,2}([\tau_i,\tau_{i+1}] \times \bar\mcO)$ with $\eta = 0$ on $\mcP_i$.

  For $Y_t^{(1)} \ne Y_t^{(2)}$ we have $a_t = \rho_2(t)\dot\Phi(\zeta_t)$ with $\z_t \in [\rho_2(t) Y_t^{(1)},\rho_2(t) Y_t^{(2)}]$ and thus $\|a\|_{L^\infty(\mcO_T)} < \infty$ by essential boundedness of $Y^{(i)}$.
  We consider a non-degenerate, smooth approximation of $a$. Set $\hat a_\ve := a \vee \ve$ and let $a_{\ve,\d}$ be a smooth approximation of $\hat a_\ve$ such that $a_{\ve,\d} \ge \ve$ and $\int_{\mcO_T} |Y|^2 (\hat a_\ve - a_{\ve,\d})^2 \ d\xi dr \le \d$. Then choose $a_{\ve} = a_{\ve,\ve^2}$.

  Let $\eta = \frac{\vp}{\rho_{1}^{(\ve)}(\tau_i)} \in C^{0,2}(\mcO_i)$ with $\vp$ being the classical solution to
  \begin{equation}\label{ra_m:eqn:vp_defn}\begin{split}
    \partial_t\vp + a_\ve \rho_{1}^{(\ve)}(\tau_i) \D\left(\frac{\rho_1^{(\ve)}}{\rho_{1}^{(\ve)}(\tau_i)} \vp\right) - \t  &= 0, \text{ on } \mcO_i \\
    \vp                                              &= 0, \text{ on } [\tau_i,\tau_{i+1}] \times \partial\mcO \\
    \vp(\tau_{i+1})                                           &= 0, \text{ on } \mcO, 
  \end{split}\end{equation}
  where $\t$ is an arbitrary, non-positive, smooth testfunction and for simplicity of notation we suppress the $\ve$-dependency of $\vp$. Time inversion transforms \eqref{ra_m:eqn:vp_defn} into a uniformly parabolic linear equation with smooth coefficients. Thus, unique existence of a non-negative classical solution follows from standard results (cf.\ e.g.\ \cite{LSU67}). 

  Consequently,
  \begin{equation}\label{ra_m:eqn:uniqueness_1}\begin{split}
    0 &\le 
     \int_{\mcO_i} Y \big( \partial_r \eta + a \D(\rho_1 \eta) \big) \ d\xi dr \\
    &= \int_{\mcO_i} Y \big( \partial_r \eta + a_{\ve} \D(\rho_1^{(\ve)} \eta) \big) \ d\xi dr + \int_{\mcO_i} Y (a - a_{\ve}) \D(\rho_1^{(\ve)} \eta)\ d\xi dr \\
       &\hskip15pt + \int_{\mcO_i} Y a \D((\rho_1-\rho_1^{(\ve)}) \eta)\ d\xi dr \\
    &= \int_{\mcO_i} \frac{1}{\rho_{1}^{(\ve)}(\tau_i)} Y \t \ d\xi dr + \int_{\mcO_i} Y (a - a_{\ve}) \D\left(\frac{\rho_1^{(\ve)}}{\rho_{1}^{(\ve)}(\tau_i)} \vp \right)  \ d\xi dr \\
        &\hskip15pt  + \int_{\mcO_i} Y a \D\left(\frac{\rho_1-\rho_1^{(\ve)}}{\rho_{1}^{(\ve)}(\tau_i)} \vp\right)\ d\xi dr.
  \end{split}\end{equation}
  We need to prove that the last two terms vanish for $\ve \to 0$. We note
  \begin{equation}\begin{split}\label{eqn:first_error}
    &\int_{\mcO_i} Y (a - a_{\ve}) \D\left(\frac{\rho_1^{(\ve)}}{\rho_{1}^{(\ve)}(\tau_i)} \vp\right)  \ d\xi dr \\
    &\hskip15pt\le C \left( \int_{\mcO_i} a_{\ve} |\D\left(\frac{\rho_1^{(\ve)}}{\rho_{1}^{(\ve)}(\tau_i)} \vp\right)|^2  \ d\xi dr \right)^\frac{1}{2}  \sqrt{\ve} \\
    &\hskip15pt\le C \left\|\frac{\rho_1^{(\ve)}}{\rho_{1}^{(\ve)}(\tau_i)} \right\|_{C^2(\mcO_i)}\left( \int_{\mcO_i} a_{\ve} |\D \vp|^2 + |\nabla \vp|^2 \ d\xi dr \right)^\frac{1}{2}  \sqrt{\ve}      
  \end{split} \end{equation}
  and
  \begin{equation}\begin{split}\label{eqn:2nd_error}
    \int_{\mcO_i} Y a \D\left(\frac{\rho_1-\rho_1^{(\ve)}}{\rho_{1}^{(\ve)}(\tau_i)} \vp\right)  \ d\xi dr
    & \le C \left\|\frac{\rho_1-\rho_1^{(\ve)}}{\rho_{1}^{(\ve)}(\tau_i)}\right\|_{H^2(\mcO_i)} \|\vp\|_{H^2(\mcO_i)} \\
    & \le C \ve^2 \|\vp\|_{H^2(\mcO_i)}.       
  \end{split} \end{equation}

  Therefore, we first derive a bound for $\|\vp\|_{H^2(\mcO_i)}$ with explicit control on the possible explosion for $\ve \to 0$. Let $\z \in C^\infty(\R)$ with $\z(\tau_i)=0$, $\z \le 1$ on $[0,T]$ and $\dot\z \ge c > 0$, for some $c \le \frac{1}{4 \g}$. Multiplying \eqref{ra_m:eqn:vp_defn} by $\z \D\vp$ and integrating yields
  \begin{equation}\begin{split}\label{eqn:dual_eqn}
      &\int_{\mcO_i} \left(\partial_r\vp \right) \z \D \vp \ d\xi dr \\
      &\hskip15pt=  \int_{\mcO_i} \left( -a_{\ve} \rho_{1}^{(\ve)}(\tau_i) \D\left(\frac{\rho_1^{(\ve)}}{\rho_{1}^{(\ve)}(\tau_i)} \vp\right) \z \D \vp + \t \z \D \vp \right) \ d\xi dr .
  \end{split}\end{equation}
  We compute
  \begin{align*}
   &-\int_{\mcO_i} a_{\ve} \rho_{1}^{(\ve)}(\tau_i) \D\left(\frac{\rho_1^{(\ve)}}{\rho_{1}^{(\ve)}(\tau_i)} \vp\right) \z \D \vp \ d\xi dr \\
   &= -\int_{\mcO_i} \z a_{\ve} \rho_{1}^{(\ve)}(\tau_i) \left(\frac{\rho_1^{(\ve)}}{\rho_{1}^{(\ve)}(\tau_i)}\right) |\D \vp|^2 \ d\xi dr \\
    &\hskip15pt + \int_{\mcO_i} \z a_{\ve} \rho_{1}^{(\ve)}(\tau_i) \left( 2\nabla \left(\frac{\rho_1^{(\ve)}}{\rho_{1}^{(\ve)}(\tau_i)}\right)\nabla \vp + \vp \D\left(\frac{\rho_1^{(\ve)}}{\rho_{1}^{(\ve)}(\tau_i)} \right)\right) \D \vp \ d\xi dr \\
   &\le - \frac{1}{4} \int_{\mcO_i}  \z a_{\ve} \rho_{1}^{(\ve)}  |\D \vp|^2 \ d\xi dr \\
   &\hskip15pt + C_1 \|\rho_{1}^{(\ve)}(\tau_i)\|_{C^0(\mcO)} \left\|\nabla \left(\frac{\rho_1^{(\ve)}}{\rho_{1}^{(\ve)}(\tau_i)}\right) \right\|_{C^{0}(\mcO_i)}^2  \int_{\mcO_i}  |\nabla \vp|^2 \ d\xi dr  \\
   &\hskip15pt + C_1 \|\rho_{1}^{(\ve)}(\tau_i)\|_{C^0(\mcO)} \left\|\D \left(\frac{\rho_1^{(\ve)}}{\rho_{1}^{(\ve)}(\tau_i)}\right) \right\|_{C^{0}(\mcO_i)}^2 \int_{\mcO_i}  |\vp|^2 \ d\xi dr \\
   &\le - \frac{1}{4} \int_{\mcO_i}  \z a_{\ve} \rho_{1}^{(\ve)}(\tau_i)  |\D \vp|^2 \ d\xi dr 
      + \frac{c}{4}  \int_{\mcO_i}  |\nabla \vp|^2 \ d\xi dr,
  \end{align*}
  where we use \eqref{eqn:choice_tau_2}. Using this in \eqref{eqn:dual_eqn} together with the arbitrariness of $\z$ with the above properties, Fatou's Lemma and strict positivity of $\rho_1^{(\ve)}(\tau_i)$ we deduce
  \begin{align*}
    \frac{c}{2} \int_{\mcO_i} |\nabla\vp|^2 \ d\xi dr + \frac{1}{4} \int_{\mcO_i} a_{\ve} |\D \vp|^2 \ d\xi dr 
    \le C \int_{\mcO_i} |\nabla \t|^2\ d\xi dr
  \end{align*}
  and $\|\vp\|_{H^2(\mcO_i)} \le \frac{C}{\ve} \int_{\mcO_i} |\nabla \t|^2 d\xi dr$ due to $a_\ve \ge \ve$. For \eqref{eqn:first_error} this implies
  \begin{equation*}
    \int_{\mcO_i} Y (a - a_{\ve}) \D\left(\frac{\rho_1^{(\ve)}}{\rho_{1}^{(\ve)}(\tau_i)} \vp \right)  \ d\xi dr 
    \le C \|\t\|_{H_0^1(\mcO_i)}^2  \sqrt{\ve}.
  \end{equation*}
  for \eqref{eqn:2nd_error}
  \begin{align*}
    \int_{\mcO_i} Y a \D\left(\frac{\rho_1-\rho_1^{(\ve)}}{\rho_{1}^{(\ve)}(\tau_i)} \vp \right)  \ d\xi dr
    & \le C \ve \|\t\|_{H_0^1(\mcO_i)}^2.
  \end{align*}

  Taking $\ve \to 0$ in in \eqref{ra_m:eqn:uniqueness_1} thus yields
  \begin{align*}
      &0 \le \int_{\mcO_i} \frac{1}{\rho_1^{(\ve)}(\tau_1)} Y \t \ d\xi dr, 
  \end{align*}
  for any non-positive, smooth testfunction $\t$. Thus $Y^{(1)} \le Y^{(2)}$ in $\mcO_i = [\tau_i,\tau_{i+1}] \times \mcO$ almost everywhere. Induction finishes the proof.

  {\it Case (A2'):} It is sufficient to prove comparison for a short time-interval $[0,\tau_1]$ for some $\tau_1 > 0$, since case (A1') may be applied on $[\tau_1,T]$ subsequently. Let $0 = \tau_0 < \tau_1$. As for case (A1') we note
  \begin{align*}
    \int_{\mcO_0} Y \big( \partial_r \eta + a \D(\rho_1 \eta) \big) \ d\xi dr \ge 0,
  \end{align*} 
  for all non-negative $\eta \in C^{1,2}([0,\tau_1] \times \bar\mcO)$ with $\eta = 0$ on $\mcP_0$.

  We follow the same idea of prove as in the case of (A1'). Hence, let $a^{(\ve)}, \rho_1^{(\ve)}$ be smooth approximations as before and $\vp$ be the classical solution to 
  \begin{equation}\label{eqn:vp_defn_2}\begin{split}
    \partial_t\vp + a_\ve  \D(\rho_1^{(\ve)}  \vp) - \t  &= 0, \text{ on } \mcO_0 \\
    \vp                                            &= 0, \text{ on } [0,\tau_{1}] \times \partial\mcO \\
    \vp(\tau_{1})                                           &= 0, \text{ on } \mcO,
  \end{split}\end{equation}
  where $\t$ is an arbitrary, non-positive, smooth testfunction. As for \eqref{ra_m:eqn:uniqueness_1} this yields
  \begin{align*}
    0 \le & \int_{\mcO_0} Y \t \ d\xi dr + \int_{\mcO_0} Y (a - a_{\ve}) \D(\rho_1^{(\ve)} \vp)  \ d\xi dr \\
          &+ \int_{\mcO_0} Y a \D((\rho_1-\rho_1^{(\ve)}) \vp)\ d\xi dr. 
  \end{align*} 
  We thus aim to show that the last two terms vanish for $\ve \to 0$. Due to the degeneracy of $\rho_1(t)$ for $t \to 0$ care has be taken in establishing the required a-priori bound on $\vp$. Multiplying \eqref{eqn:vp_defn_2} by $\z \D \vp$ as before and noting
    $$ \D(\rho_1^{(\ve)} \vp) = \rho_1^{(\ve)} \D\vp + 2 \nabla \rho_1^{(\ve)} \cdot \nabla \vp + \vp \D \rho_1^{(\ve)}$$
  we obtain
  \begin{align*}
    &\int_{\mcO_0}  \left(\partial_r\vp\right) \z \D \vp\ d\xi dr \\
    &=  \int_{\mcO_0} \left( -a_{\ve} \D(\rho_1^{(\ve)} \vp) \z \D \vp + \t \z \D \vp \right) \ d\xi dr \\
    &\le -\frac{1}{2} \int_{\mcO_0} a_{\ve} \rho_1^{(\ve)} \z |\D \vp|^2 \ d\xi dr \\
    &\hskip15pt +  \left(C_1 \left\|\frac{|\nabla \rho_1^{(\ve)}|^2}{\rho_1^{(\ve)}}\right\|_{C^0(\mcO_0)} + C_1 \left\| \frac{|\D\rho_1^{(\ve)}|^2}{\rho_1^{(\ve)}} \right\|_{C^0(\mcO_0)} + \frac{c}{4} \right) \int_{\mcO_0} |\nabla \vp|^2 \ d\xi dr \\
    &\hskip15pt +  C \int_{\mcO_0} |\nabla \t|^2 \ d\xi dr,
  \end{align*}
  where $C_1$ is a constant depending only on $\|a\|_{L^\infty(\mcO_T)}$ and $c > 0$ is a constant as in case (A1'). We now choose $\tau_1 > 0$ such that
    $$C_1\left(\left\|\frac{|\nabla \rho_1^{(\ve)}|^2}{\rho_1^{(\ve)}}\right\|_{C^0(\mcO_0)} + \left\| \frac{|\D\rho_1^{(\ve)}|^2}{\rho_1^{(\ve)}} \right\|_{C^0(\mcO_0)} \right) \le \frac{c}{4}.$$
  By the choice of $\tau_1$ and Fatou's Lemma we get
  \begin{align*}
    \frac{c}{2} \int_{\mcO_0} |\nabla\vp|^2 \ d\xi dr + \frac{1}{2} \int_{\mcO_0} a_{\ve} \rho_1^{(\ve)} |\D \vp|^2 \ d\xi dr \le  C \int_{\mcO_0} |\nabla \t|^2 \ d\xi dr
  \end{align*}
  and we conclude the proof as in case (A1').
\end{proof}

\subsection{Lower bound on \texorpdfstring{$L^1$}{L1}-decay for \texorpdfstring{\eqref{eqn:general_trans}}{(\ref{eqn:general_trans})}} In this section we provide a lower bound for the decay of the $L^1$ norm of solutions to \eqref{eqn:general_trans}. This estimate is required in Section \ref{sec:ra} in order to ensure that the constructed solutions with disjoint support have a sufficiently large distance with respect to the $L^1$-norm.

\begin{proposition}\label{prop:lower_l1}
   Let $Y \in C((0,T]\times\mcO)$ be an essentially bounded, non-negative, very weak supersolution to the homogeneous Cauchy-Dirichlet problem for \eqref{eqn:general_trans} with uniformly compact support, i.e.\ $ K := \bigcup_{t \in [0,T]} \supp(Y_t) \Subset \mcO$ is a precompact set. Then, there is a constant $C > 0$ such that
      $$ \|Y_t\|_{L^1(\mcO)} \ge e^{-Ct\|Y\|_{L^\infty(\mcO_T)}^{m-1}} \|Y_0\|_{L^1(\mcO)}, \quad \forall t \in [0,T]. $$
\end{proposition}
\begin{proof}
   It is easy to see that $Y$ is a very weak supersolution to the homogeneous Cauchy-Dirichlet problem for \eqref{eqn:very_weak_rough_transformed} iff
  \begin{equation}\label{eqn:very_weak_rough_transformed_2}\begin{split}
        \int_{\mcO} Y_t \vp\ d\xi - \int_\mcO Y_s \vp \ d\xi
        \le &-\int_s^t \int_{\mcO} \Phi(\rho_2 Y) \D (\rho_1 \vp )\ d\xi dr, \quad \forall 0 \le s < t \le T,
   \end{split}\end{equation}            
  and all non-negative $\vp \in C^{2}(\bar\mcO)$ with $\vp_{|\partial\mcO} = 0$. Let $M := \|Y\|_{L^\infty(\mcO_T)}$. We choose a test-function $\vp \in C^2_c(\mcO)$ with $\vp \equiv 1$ on $K$. Then
   \begin{align*}
      \|Y_t\|_{L^1(\mcO)} 
      &= \int_{\mcO} Y_t \vp d\xi  \\
      &\le \int_{\mcO} Y_s \vp  d\xi - \int_s^t \int_{\mcO} \Phi(\rho_2 Y) \D(\rho_1 \vp ) d\xi dr\\
      &= \|Y_s\|_{L^1(\mcO)} - M^{m-1} \int_s^t \int_\mcO \Phi(\rho_2) \left(\frac{Y}{M}\right)^{m-1} |Y| \D\rho_1 d\xi dr\\
      &\ge \|Y_s\|_{L^1(\mcO)} - M^{m-1}(\|\rho_1\|_{C^{0,2}(\mcO_T)}+\|\rho_2^m\|_{C^{0}(\mcO_T)}) \int_s^t \|Y_r\|_{L^1(\mcO)} dr,
   \end{align*}
  for all $0 \le s < t \le T$. Gronwall's Lemma finishes the proof.
\end{proof}

\subsection{Finite speed of propagation for \texorpdfstring{\eqref{eqn:general_trans}}{(\ref{eqn:general_trans})}}\label{sec:gen_expansion} The proof of finite speed of propagation for \eqref{eqn:general_trans} is very similar to Theorem \ref{thm:propagation} and is based on a bound for the speed of hole-filling as given in Theorem \ref{thm:hole_filling_local_time}. The arguments remain the same with minor changes in the calculation. For the readers convenience we state the corresponding results in detail and give some short remarks on the proofs.

\begin{theorem}\label{thm:gen_hole_filling}
Let $\xi_0 \in \R^d$, $T,R  > 0$ and $Y \in C((0,T] \times B_R(\xi_0))$ be an essentially bounded, non-negative, very weak subsolution to \eqref{eqn:general_trans} with vanishing initial value $Y_0$ on $B_R(\xi_0)$ and boundary value $g$ satisfying  $H := \|g\|_{L^\infty([0,T] \times \partial B_R(\xi_0))}  < \infty$.
  Define $T_{stoch}$ by
    $$ T_{stoch} := \sup\Bigg\{\td T \in [0,T] \Big|\  \td T C_{\td T} \le  R^2 H^{-(m-1)}\Bigg\} ,$$
  where $t \mapsto C_t$ is a continuous, non-decreasing function.

  Then $Y_t$ vanishes in $B_{R_{stoch}(t)}(\xi_0)$ for all $t \in [0,T_{stoch}]$, where
     $$R_{stoch}(t) = R - \sqrt{t} \sqrt{C_t} H^\frac{m-1}{2}.$$
\end{theorem}
\begin{proof}
   As in Theorem \ref{thm:hole_filling_local_time} the proof is based on the construction of an appropriate supersolution to \eqref{eqn:general_trans}. For  $r \in (0,R]$, $\xi_1 \in \R^d$, $\td T > 0$ let
    $$W(t,\xi,\xi_1) := \td C |\xi-\xi_1|^\frac{2}{m-1}\left(\td T-t\right)^{\frac{-1}{m-1}}, \quad t \in [0,\td T), \xi \in B_r(\xi_1).$$
   Direct computations yield
    $$  \partial_t W(t,\xi,\xi_1) \ge \rho_1 \D \left(\rho_2 W(t,\xi,\xi_1)\right)^m $$
on $[0,\td T) \times B_r(\xi_1)$ if
    $$ 1 \ge C(d,m) \td C^{m-1} (1 + R)^2 \|\rho_1\|_{C^{0}([0,\td T] \times B_R(\xi_0))} \|\rho_2\|_{C^{0,2}([0,\td T] \times B_R(\xi_0))},$$
  for all $(t,\xi) \in [0,\td T) \times B_r(\xi_1)$ and some generic constant $C(d,m)$. This is satisfied for the choice
  \begin{align*}
    \td C^{m-1} = \td C_{\td T}^{m-1} := \left(C(d,m)(1 + R)^2 \|\rho_1\|_{C^{0}([0,\td T] \times B_R(\xi_0))} \|\rho_2\|_{C^{0,2}([0,\td T] \times B_R(\xi_0))}\right)^{-1}.
  \end{align*}

  Moreover,
  \begin{equation*}
     W(t,\xi,\xi_1) = \td C |\xi-\xi_1|^\frac{2}{m-1}\left(\td T-t\right)^{\frac{-1}{m-1}} \ge H,
  \end{equation*}
  for a.a.\ $(t,\xi) \in [0,\td T) \times \partial B_r(\xi_1)$ is satisfied if 
  \begin{equation*}
      \td T \td C_{\td T}^{-(m-1)} \le  r^2 H^{-(m-1)} 
  \end{equation*}
  We conclude the proof as for Theorem \ref{thm:hole_filling_local_time}.
\end{proof}

As in Theorem \ref{thm:propagation_time} we may now use Theorem \ref{thm:gen_hole_filling} to deduce

\begin{theorem}\label{thm:gen_expansion}
  Let $Y \in C((0,T]\times\mcO)$ be an essentially bounded, non-negative, very weak subsolution to the homogeneous Dirichlet problem to \eqref{eqn:SPME} and set $H := \|Y\|_{L^\infty(\mcO_T)}$. Then, for every $ s \in [0,T]$ 
      $$ supp (Y_{s+t}) \subseteq  B_{\sqrt{t}\sqrt{C_{t}} H^{\frac{m-1}{2}}}(\supp(Y_s)),\quad \forall t \in [0,T-s],$$
    where $t \mapsto C_{t}$ is a continuous, non-decreasing function.
\end{theorem}


\bibliographystyle{plain}
\bibliography{../../latex-refs/refs}

\def\cprime{$'$}
\begin{thebibliography}{10}

\bibitem{A98}
Ludwig Arnold.
\newblock {\em Random dynamical systems}.
\newblock Springer Monographs in Mathematics. Springer-Verlag, Berlin, 1998.

\bibitem{BDPR08}
Viorel Barbu, Giuseppe Da~Prato, and Michael R{\"o}ckner.
\newblock Existence and uniqueness of nonnegative solutions to the stochastic
  porous media equation.
\newblock {\em Indiana Univ. Math. J.}, 57(1):187--211, 2008.

\bibitem{BDPR08-2}
Viorel Barbu, Giuseppe Da~Prato, and Michael R{\"o}ckner.
\newblock Some results on stochastic porous media equations.
\newblock {\em Boll. Unione Mat. Ital. (9)}, 1(1):1--15, 2008.

\bibitem{BDPR09}
Viorel Barbu, Giuseppe Da~Prato, and Michael R{\"o}ckner.
\newblock Existence of strong solutions for stochastic porous media equation
  under general monotonicity conditions.
\newblock {\em Ann. Probab.}, 37(2):428--452, 2009.

\bibitem{BDPR09-2}
Viorel Barbu, Giuseppe Da~Prato, and Michael R{\"o}ckner.
\newblock Stochastic porous media equations and self-organized criticality.
\newblock {\em Comm. Math. Phys.}, 285(3):901--923, 2009.

\bibitem{BDPR12}
Viorel Barbu, Giuseppe Da~Prato, and Michael R{\"o}ckner.
\newblock Finite time extinction of solutions to fast diffusion equations
  driven by linear multiplicative noise.
\newblock {\em J. Math. Anal. Appl.}, 389(1):147--164, 2012.

\bibitem{BR11b}
Viorel Barbu and Michael R{\"o}ckner.
\newblock On a random scaled porous media equation.
\newblock {\em J. Differential Equations}, 251(9):2494--2514, 2011.

\bibitem{BR12}
Viorel Barbu and Michael R{\"o}ckner.
\newblock Localization of solutions to stochastic porous media equations:
  finite speed of propagation.
\newblock {\em Electronic Journal of Probability}, 16:1--11, 2012.

\bibitem{BR11}
Viorel Barbu and Michael R{\"o}ckner.
\newblock Stochastic porous media equations and self-organized criticality:
  convergence to the critical state in all dimensions.
\newblock {\em Comm. Math. Phys.}, 311(2):539--555, 2012.

\bibitem{BGLR10}
Wolf-J\"urgen Beyn, Benjamin Gess, Paul Lescot, and Michael R{\"o}ckner.
\newblock The global random attractor for a class of stochastic porous media
  equations.
\newblock {\em Comm. Partial Differential Equations}, 36(3):446--469, 2011.

\bibitem{BL06}
Zdzislaw Brze{\'z}niak and Yuhong Li.
\newblock Asymptotic compactness and absorbing sets for 2d stochastic
  navier-stokes equations on some unbounded domains.
\newblock {\em Trans. Amer. Math. Soc.}, 358(12):5587--5629, 2006.

\bibitem{CCLR07}
Tom{\'a}s Caraballo, Hans Crauel, Jos{\'e}~A. Langa, and James~C. Robinson.
\newblock The effect of noise on the {C}hafee-{I}nfante equation: a nonlinear
  case study.
\newblock {\em Proc. Amer. Math. Soc.}, 135(2):373--382 (electronic), 2007.

\bibitem{CR04}
Tom{\'a}s Caraballo and James~C. Robinson.
\newblock Stabilisation of linear {PDE}s by {S}tratonovich noise.
\newblock {\em Systems Control Lett.}, 53(1):41--50, 2004.

\bibitem{CDF97}
Hans Crauel, Arnaud Debussche, and Franco Flandoli.
\newblock Random attractors.
\newblock {\em J. Dynam. Differential Equations}, 9(2):307--341, 1997.

\bibitem{CF94}
Hans Crauel and Franco Flandoli.
\newblock Attractors for random dynamical systems.
\newblock {\em Probab. Theory Related Fields}, 100(3):365--393, 1994.

\bibitem{DPR04}
Giuseppe Da~Prato and Michael R{\"o}ckner.
\newblock Weak solutions to stochastic porous media equations.
\newblock {\em J. Evol. Equ.}, 4(2):249--271, 2004.

\bibitem{DPRRW06}
Giuseppe Da~Prato, Michael R{\"o}ckner, Boris~L. Rozovskii, and Feng-Yu Wang.
\newblock Strong solutions of stochastic generalized porous media equations:
  existence, uniqueness, and ergodicity.
\newblock {\em Comm. Partial Differential Equations}, 31(1-3):277--291, 2006.

\bibitem{DB83}
Emmanuele DiBenedetto.
\newblock Continuity of weak solutions to a general porous medium equation.
\newblock {\em Indiana Univ. Math. J.}, 32(1):83--118, 1983.

\bibitem{EZ08}
Messoud~A. Efendiev and Sergey~V. Zelik.
\newblock Finite- and infinite-dimensional attractors for porous media
  equations.
\newblock {\em Proc. Lond. Math. Soc. (3)}, 96(1):51--77, 2008.

\bibitem{G11c}
Benjamin Gess.
\newblock Random attractors for stochastic porous media equations perturbed by
  space-time linear multiplicative noise.
\newblock {\em arXiv:1108.2413v1}, 2011.

\bibitem{G11b}
Benjamin Gess.
\newblock Random attractors for degenerate stochastic partial differential
  equations.
\newblock {\em arXiv:1206.2329}, 2012.

\bibitem{G11c-2}
Benjamin Gess.
\newblock Random attractors for stochastic porous media equations perturbed by
  space-time linear multiplicative noise.
\newblock {\em Comptes Rendus Mathematique}, 350(5--6):299--302, 2012.

\bibitem{G11}
Benjamin Gess.
\newblock Strong solutions for stochastic partial differential equations of
  gradient type.
\newblock {\em J. Funct. Anal.}, 263(8):2355--2383, 2012.

\bibitem{GLR11}
Benjamin Gess, Wei Liu, and Michael R{\"o}ckner.
\newblock Random attractors for a class of stochastic partial differential
  equations driven by general additive noise.
\newblock {\em J. Differential Equations}, 251(4-5):1225 -- 1253, 2011.

\bibitem{K06}
Jong~U. Kim.
\newblock On the stochastic porous medium equation.
\newblock {\em J. Differential Equations}, 220(1):163--194, 2006.

\bibitem{LSU67}
Olga~A. Lady\v{z}enskaja, Vsevolod~A. Solonnikov, and Nina~N. Ural'ceva.
\newblock {\em Linear and quasilinear equations of parabolic type}.
\newblock Translated from the Russian by S. Smith. Translations of Mathematical
  Monographs, Vol. 23. American Mathematical Society, Providence, R.I., 1967.

\bibitem{OKY58}
Olga~A. Ole{\u\i}nik, Anatoli{\u\i}~S. Kala{\v{s}}inkov, and
  Yu{\u\i}-Lin{\cprime} {\v{C}}{\v{z}}ou.
\newblock The {C}auchy problem and boundary problems for equations of the type
  of non-stationary filtration.
\newblock {\em Izv. Akad. Nauk SSSR. Ser. Mat.}, 22:667--704, 1958.

\bibitem{RRW07}
Jiagang Ren, Michael R{\"o}ckner, and Feng-Yu Wang.
\newblock Stochastic generalized porous media and fast diffusion equations.
\newblock {\em J. Differential Equations}, 238(1):118--152, 2007.

\bibitem{RW08}
Michael R{\"o}ckner and Feng-Yu Wang.
\newblock Non-monotone stochastic generalized porous media equations.
\newblock {\em J. Differential Equations}, 245(12):3898--3935, 2008.

\bibitem{RW11}
Michael R{\"o}ckner and Feng-Yu Wang.
\newblock General extinction results for stochastic partial differential
  equations and applications.
\newblock {\em BiBoS-Preprint}, 11-10-389:1--19, 2011.

\bibitem{S92}
Bj{\"o}rn Schmalfuss.
\newblock Backward cocycle and attractors of stochastic differential equations.
\newblock In V.~Reitmann, T.~Riedrich, and N.~Koksch, editors, {\em
  International Seminar on Applied Mathematics - Nonlinear Dynamics: Attractor
  Approximation and Global Behavior}, pages 185--192. Technische Universit\"at
  Dresden, 1992.

\bibitem{V84}
Juan~L. V{\'a}zquez.
\newblock The interfaces of one-dimensional flows in porous media.
\newblock {\em Trans. Amer. Math. Soc.}, 285(2):717--737, 1984.

\bibitem{V07}
Juan~L. V{\'a}zquez.
\newblock {\em The porous medium equation}.
\newblock Oxford Mathematical Monographs. The Clarendon Press Oxford University
  Press, Oxford, 2007.
\newblock Mathematical theory.

\end{thebibliography}

\end{document}